\documentclass[11pt]{amsart}
\usepackage{palatino}
\usepackage{amsfonts}
\usepackage{amssymb}
\usepackage{amscd}

\theoremstyle{plain}

\newtheorem{theorem}{Theorem}[section]
\newtheorem{proposition}[theorem]{Proposition}
\newtheorem{corollary}[theorem]{Corollary}
\newtheorem{lemma}[theorem]{Lemma}
\theoremstyle{definition}
\newtheorem{definition}[theorem]{Definition}
\newtheorem{remark}[theorem]{Remark}
\newtheorem{example}[theorem]{Example}

\newtheorem{conjecture}[theorem]{Conjecture}
\newtheorem{conjecture/question}[theorem]{Conjecture/Question}
\newtheorem{question}[theorem]{Question}
\newtheorem{remark/definition}[theorem]{Remark/Definition}
\newtheorem{terminology/notation}[theorem]{Terminology/Notation}
\newtheorem{problem}[theorem]{Problem}
\newcommand{\marginlabel}[1]%
  {\mbox{}\marginpar{\raggedleft\hspace{0pt}\bfseries\sf#1}}

\def\rmapdown#1{\Big\downarrow
   \rlap{$\vcenter{\hbox{$\scriptstyle#1$}}$ }}

\def\GG{{\textbf G}}

\def\PP{{\textbf P}}
\def\OO{\mathcal{O}}

\def\cB{\mathcal{B}}
\def\cA{\mathcal{A}}
\def\F{\mathcal{F}}
\def\P{\mathcal{P}}

\def\E{\mathcal{E}}
\def\G{\mathcal{G}}
\def\K{\mathcal{K}}
\def\L{\mathcal{L}}
\def\I{\mathcal{I}}

\def\cM{\mathcal{M}}
\def\cR{\mathcal{R}}
\def\rr{\overline{\mathcal{R}}}
\def\cZ{\mathcal{Z}}
\def\cU{\mathcal{U}}
\def\cC{\mathcal{C}}
\def\H{\mathcal{H}}
\def\Pic0{{\rm Pic}^0(X)}

\def\mm{\overline{\mathcal{M}}}
\def\KK{\overline{\mathcal{K}}}
\def\zz{\overline{\mathcal{Z}}}
\def\rem{\overline{\textbf{M}}}

\def\pem{\widetilde{\textbf{M}}}
\def\tem{\textbf{{M}}^p}
\def\ttem{\overline{\textbf{M}}^p}

\theoremstyle{remark}

\begin{document}

\title{Birational aspects of the geometry of $\mm_g$}

\author[G. Farkas]{Gavril Farkas}
\address{Humboldt Universit\"at zu Berlin, Institut f\"ur Mathematik,
10099 Berlin} \email{{\tt farkas@math.hu-berlin.de}}
\thanks{Research  partially supported by an Alfred P. Sloan Fellowship and the NSF Grant DMS-0500747. Work on this paper was started during a
stay at the Mittag-Leffler Institute in Djursholm in 2007.
}

\maketitle

\section{Introduction}
The study of the moduli space $\cM_g$ begins of course with Riemann,
who in 1857 was the first to consider a space whose points correspond to isomorphism
classes of smooth curves of genus $g$. By viewing curves as branched covers of $\PP^1$,
Riemann correctly computed the number of \emph{moduli}, that is he showed that
$$\mbox{dim}(\cM_g)=3g-3$$ for all $g\geq 2$. Riemann is also responsible for the term
moduli, meaning essential parameters for varieties of certain kind: "\emph{... es h\"angt also eine Klasse von
Systemen gleichverzweigter $2p+1$ fach zusammenhangender Funktionen und die zu ihr
geh\"orende Klassen algebraischer Gleichungen von $3p-3$ stetig ver\"anderlichen Gr\"ossen ab,
welche die Moduln dieser Klasse werden sollen}".  The best modern way of
reproving Riemann's result is via Kodaira-Spencer deformation
theory.
The first rigorous construction of $\cM_g$ was carried out by Mumford in
1965, in the book \cite{GIT}. By adapting Grothendieck's "functorial ideology", Mumford,
used Geometric Invariant Theory and developed a purely algebraic approach to study $\cM_g$. He
indicated that one has to study the \emph{coarse moduli scheme} that
is as close as any scheme can be to the \emph{moduli stack} of smooth
curves: Although the coarse moduli scheme exists over
$\mbox{Spec}(\mathbb Z)$, one has to pass to an algebraically closed
field $k$ to get a bijection between $\mbox{Hom}(\mbox{Spec}(k),
\cM_g)$ and isomorphism classes of smooth curves of genus $g$
defined over $k$.

Despite the fact that the rigorous construction of $\cM_g$ was achieved so late, various geometric properties of the space $\cM_g$, whose existence was somehow taken for granted,  have been established.
Hurwitz \cite{Hu} following earlier work of Clebsch, proved in 1891 that $\cM_g$ is irreducible by using Riemann's existence theorem and showing that the
space parameterizing branched covers of $\PP^1$ having fixed degree
and genus is connected. In 1915,  Severi \cite{S} used plane models of curves to prove that when $g\leq 10$ the space $\cM_g$ is unirational.
For $g\leq 10$ (and only in this range), it is possible to
realize a general curve $[C]\in \cM_g$ as a nodal plane curve $\Gamma \subset \PP^2$
 having minimal degree $d=[(2g+8)/3]$, such that the nodes of
$\Gamma$ are general points in $\PP^2$. In the same paper Severi
 conjectures that $\cM_g$ is unirational (or even rational!) for all $g$. This would correspond
to being able to write down the general curve of genus $g$ explicitly, in a family depending
on $3g-3$ \emph{free} parameters. Severi himself and later B. Segre made several attempts to prove the conjecture for $g\geq 11$
using curves of minimal degree in $\PP^r$ with $r\geq 3$, cf. \cite{Seg}, \cite{God}.

Severi's Conjecture seemed plausible and was widely believed until the 1980s:
In \cite{M3} Mumford declares "How rational is the moduli space of curves" to be one of the main problems of
present day mathematics. In "Curves and their Jacobians" \cite{M2} Mumford elaborates: "\emph{Whether more
$\cM_g$'s $g\geq 11$ are unirational or not is a very interesting problem, but one which
looks very hard too, especially if $g$ is quite large}". Probably thinking by analogy with the well-understood case of moduli  of elliptic curves (with level structure), Oort formulates in his 1981 survey \cite{O} a principle that naturally defined moduli spaces should be unirational: "\emph{... generally speaking it seems that coarse moduli spaces tend to be close to rational
varieties while high up in the tower of fine moduli spaces, these varieties possibly are of general
type}".

It came as a major surprise
when in 1982 Harris and Mumford \cite{HM} showed that Severi's Conjecture is maximally wrong in the sense that $\cM_g$ itself rather than a higher level cover is almost always a variety of general type!
\begin{theorem}\label{hme}
For $g\geq 24$, the moduli space of stable curves $\mm_g$ is a variety of general
type.
\end{theorem}
An easy consequence of Theorem \ref{hme} is the following negative result:
\begin{corollary}
For $g\geq 24$, if $[C]\in \cM_g$ is a general curve and $S$ is a surface containing $C$ on which
$C$ moves in a non-trivial linear system, then $S$ is birational to $C\times \PP^1$. A general curve of genus $g\geq 24$ does not occur in any non-trivial linear system on any non-ruled surface.
\end{corollary}  The proof of Theorem \ref{hme} uses
in an essential way the Deligne-Mumford compactification $\mm_g$ by means of stable curves. The key idea is to reduce the problem of producing pluricanonical
forms on $\mm_g$ to a divisor class calculation on $\mm_g$. For instance, in the case $g=2k-1$, Harris and Mumford consider the Hurwitz divisor
$$\cM_{g, k}^1:=\{[C]\in \cM_g: \exists \ \  C\stackrel{k:1}\rightarrow \PP^1\}.$$
By computing the class of the closure $\mm_{g, k}^1$ of $\cM_{g, k}^1$ inside $\mm_g$, it follows that for $g=2k-1\geq 25$, the canonical class $K_{\mm_g}$ lies in the cone spanned inside $\mathrm{Pic}(\mm_g)_{\mathbb Q}$ by $[\mm_{g, k}^1]$, the Hodge class $\lambda\in \mbox{Pic}(\mm_g)$ and the irreducible components of the boundary $\mm_g-\cM_g$.  Since the class $\lambda$ is big, that is, high multiples of $\lambda$ have the maximal number of sections, it follows that high multiples of $K_{\mm_g}$ will also have the maximum number of sections, that is, $\mm_g$ is of general type. The main technical achievement of \cite{HM} is the calculation of the class $[\mm_{g, k}^1]$ via the theory of \emph{admissible coverings}. The case of even genus was initially settled in \cite{H1} for $g=2k-2\geq 40$ and later
 greatly simplified and improved by Eisenbud and Harris \cite{EH3} via the theory of \emph{limit linear series}. In this survey, apart from reviewing the work of Harris, Mumford and Eisenbud, we present a different  proof of
Theorem \ref{hme} by replacing the divisor $\cM_{g, k}^1$ by a Koszul divisor on $\mm_g$ in the spirit of \cite{F3}.
It turns out that modulo Voisin's proof \cite{V2} of the generic Green Conjecture on syzygies of canonical curves, one obtains a very short proof of the Harris-Mumford Theorem \ref{hme}, which does not resort to enumerative calculation on Hurwitz stacks of admissible coverings or to limit linear series.

After \cite{HM} there has been a great deal of work trying to describe the geometry of $\cM_g$ in the intermediate
cases $11\leq g\leq 23$.
Extending Severi's result to genera $g\geq 11$ requires subtle
ideas and the use of powerful modern techniques, even though the idea
of the proof is simple enough.  Sernesi \cite{Se1} was
the first to go past the classical analysis of Severi by proving
that $\cM_{12}$ is unirational. A few years later,
M. C. Chang and Z. Ran proved that $\cM_{11}$ and $\cM_{13}$ are
also unirational (cf. \cite{CR1}). In the process, they gave another
proof for Sernesi's theorem for $\cM_{12}$. The case $g=14$ remained
open for a long time, until Verra, using liaison techniques as well as Mukai's work on models of canonical curves of genus at most $9$, proved that $\cM_{14}$ is
unirational.  Verra's approach gives a much simpler proof of the unirationality of $\cM_g$ in the cases $g=11, 12, 13$ as well. We shall explain his main ideas following \cite{Ve}.

Chang and Ran showed that $\kappa(\mm_g)=-\infty$ for $g=15, 16$, cf. \cite{CR2}, \cite{CR3}.
This was recently improved by Bruno and Verra \cite{BV} who proved that $\cM_{15}$ is rationally connected. Precisely,  they proved that a general curve $[C]\in \cM_{15}$ embedded via a linear series $C\stackrel{|L|}\longrightarrow \PP^6$, where $L\in W^6_{19}(C)$, lies on a smooth complete intersection surface $S\subset \PP^6$ of type $(2, 2, 2, 2)$, in such a way that $\mbox{dim } |\OO_S(C)|=1$. This last statement follows via a standard exact sequence argument because such a surface $S$ is canonical.

Turning to genus $16$, it is proved in \cite{CR3} that $K_{\mm_{16}}$ is not a pseudo-effective class. It follows from \cite{BDPP}, that this actually implies that $\mm_{16}$ is uniruled.\footnote{More generally, it follows that $\mm_g$ is uniruled whenever one can show that $K_{\mm_g}$ is not a pseudo-effective class. I am grateful to J. McKernan for pointing this out to me.} The question whether
$\mm_{15}$ or $\mm_{16}$ are actually unirational remains open and seems difficult. Note that the above mentioned argument from \cite{BV} actually implies that through a general point of $\mm_{15}$ there passes a rational surface.
\begin{question}
What is the Kodaira dimension of $\mm_g$ for $17\leq g\leq 21$?
\end{question}
A partial result for $\mm_{23}$ was obtained in \cite{F1} where the inequality $\kappa(\mm_{23})\geq 2$ is proved.
Section 7 of this paper is devoted to the proof of the following result:
\begin{theorem}
The moduli space $\mm_{22}$ is of general type.
\end{theorem}

Similar questions about the birational type of other moduli spaces have been studied. Logan \cite{Log} has proved that for all $4\leq g\leq 22$ there exists an explicitly known integer $f(g)$ such that $\mm_{g, n}$ is of general type for $n\geq f(g)$. The bounds on the function $f(g)$ have been significantly improved in \cite{F3}. The moduli space $\cA_g$ of
principally polarized abelian varieties of dimension $g$ is known to be of general type for $g\geq 7$ due to results
of Freitag \cite{Fr}, Mumford \cite{M5} and Tai \cite{T} (For a comprehensive recent review of developments on the global geometry of $\cA_g$, see \cite{Gru}). Freitag was the first to go beyond the classical picture and show that for $g\geq 17, \ g\equiv 1 \mbox{ mod }8$, the space $\cA_g$ rather than one of its covers corresponding to "moduli with level structure"
is of general type. Freitag's work seems to have been essential in making Mumford realize that Severi's Conjecture
might be fundamentally false, see the discussion in \cite{HM} pg. 24. We mention that using e.g. the moduli space of Prym varieties, one can show that $\cA_g$ is known to be unirational for $g\leq 5$, cf. \cite{Don}, \cite{Ve2}. The remaining question is certainly difficult and probably requires new ideas:
\begin{question}
What is the Kodaira dimension of $\cA_6$?
\end{question}

Tai also discovered an important criterion (now called the \emph{Reid-Shepherd-Barron- Tai} criterion) for canonical forms on the smooth locus of spaces with finite quotient singularities to extend to any resolution of singularities. He then showed that $\cA_g$ satisfies the Reid-Shepherd-Barron-Tai criterion. A similar analysis of singularities (which is needed whenever one shows that a coarse moduli space is of general type), in the case of
$\mm_g$, has been carried out in \cite{HM} Theorem 1.

A very interesting moduli space (also in light of Section 6 of this paper and the connection with the Slope Conjecture), is the moduli space $\mathcal{F}_g$ of polarized $K3$ surfaces $[S, h]$ of degree $h^2=2g-2$. On $\mathcal{F}_g$ one considers the $\PP^g$-bundle
$$\mathcal{P}_g:=\{\bigl([S, h], C\bigr): [S, h]\in \mathcal{F}_g, C\in |h|\}$$ together with the projections $p_1:\mathcal{P}_g\rightarrow \mathcal{F}_g$ and $p_2:\mathcal{P}_g --> \cM_g$. The image $\mathcal{K}_g:=p_2(\mathcal{P}_g)$ is the locus of curves that can be abstractly embedded in a $K3$ surface. For $g\geq 13$ the map $p_2$ is generically finite (in fact, generically injective cf. \cite{CLM}), hence $\mbox{dim}(\mathcal{K}_g)=19+g$. This locus appears as an obstruction for an effective divisor on $\mm_g$ to have small slope, cf. Proposition \ref{nef}.
The geometry of $\mathcal{F}_g$ has been studied
in low genus by Mukai and in general, using automorphic form techniques, initially by Kondo \cite{K} and more recently, to great effect, by Gritsenko, Hulek and Sankaran \cite{GHS}. Using Borcherds's construction of automorphic forms on locally symmetric domains of type IV, they proved that (any suitable compactification of) $\mathcal{F}_g$ is of general type for $g>62$ as well as for $g=47, 51, 55, 58, 59, 61$. The largest $g$ for which $\mathcal{F}_g$ is known to be unirational is equal to $20$, cf. \cite{M4}.

\begin{problem}
 Prove  purely algebro-geometrically that $\mathcal{F}_g$ is of general type for $g$ sufficiently large. Achieve this by computing the class of a geometric (Noether-Lefschetz, Koszul) divisor on $\F_g$ and comparing this calculation against the canonical class.
\end{problem}

More generally, it is natural to ask whether the time is ripe for a systematic study of the birational invariants of the Alexeev-Koll\'ar-Shepherd-Barron moduli spaces of higher dimensional  varieties (see \cite{AP}, \cite{H} for a few beautiful, yet isolated examples when the geometry of such spaces has been completely worked out).

We end this discussion by describing the birational geometry of the moduli space $\cR_g$ classifying pairs $[C, \eta]$ where $[C]\in \cM_g$ and $\eta\in \mbox{Pic}^0(C)[2]$ is a point of order $2$ in its Jacobian. This moduli space provides an interesting
correspondence between $\cM_g$ and $\cA_{g-1}$ via the natural projection $\pi:\cR_g\rightarrow \cM_g$ and the \emph{Prym map} $$\mathfrak{Pr}_g: \cR_g\rightarrow \cA_{g-1}.$$ For $g\leq 6$ the Prym map is dominant, thus a study of the birational invariants of $\cR_g$ gives detailed information about $\cA_{g-1}$ as well. For $g\geq 7$ the Prym map $\mathfrak{Pr}_g$ is generically injective (though never injective) and
we view $\cR_g$ as a desingularization of the moduli space of Prym varieties $\mathfrak{Pr}_g(\cR_g)\subset \cA_{g-1}$. There is a good compactification $\rr_g$ of $\cR_g$, by taking $\rr_g$ to be the coarse moduli space associated to the moduli stack of stable
maps $\rem_g(\mathcal{B}\mathbb Z_2)$. Note that the Galois covering $\pi$ extends to a finite ramified covering $\pi:\rr_g\rightarrow \mm_g$. We have the following result \cite{FL}:
\begin{theorem}\label{prijm}
The compact moduli space of Pryms $\rr_g$ is of general type for $g>13$ and $g\neq 15$. The Kodaira dimension of $\rr_{15}$ is at least $1$.
\end{theorem}
Thus there are genera (e.g. $g=14$) for which $\mm_g$ is unirational but $\rr_g$ is of general type. Note that $\rr_g$ is unirational
for $g\leq 7$ and it appears to be difficult to extend the range of $g$ for which $\rr_g$ is unirational much further.
An essential ingredient in the proof of Theorem \ref{prijm} is the analysis of the singularities of $\rr_g$. Kodaira-Spencer
theory shows that singularities of $\rr_g$ correspond to automorphisms of Prym curves. A delicate local
analysis shows that, even though the Reid-Shepherd-Barron-Tai criterion does not hold everywhere on $\rr_g$ (precisely, there is a codimension $2$ locus of non-canonical singularities),
for $g\geq 4$ every pluricanonical form defined on the smooth part of $\rr_g$  extends to any desingularization. Equivalently, for any resolution of singularities
$\epsilon: \widehat{R}_g\rightarrow \rr_g$ and $l\geq 0$, there is an isomorphism of groups
$$\epsilon^*: H^0(\rr_{g, reg}, K_{\rr_g}^{\otimes l})\stackrel{\cong}\longrightarrow H^0(\widehat{R}_g, K_{\widehat{R}_g}^{\otimes l}).$$

Since $\rr_g=\mm_g(\mathcal{B} \mathbb Z_2)$, it makes sense to raise the following more general question:
\begin{problem}
For a finite group $G$, study the birational invariants (Kodaira dimension and singularities, Picard groups, cones of ample and effective divisors) of the moduli spaces of twisted stable maps
$\mm_g(\mathcal{B} G)$.
\end{problem}

We close by outlining the structure of the paper. In Section 2 we describe various attempts to prove that $\cM_g$ is unirational, starting with Severi's classical proof when $g\leq 10$ and concluding with Verra's recent work on $\cM_g$ for $g\leq 14$. While our presentation follows  \cite{Ve}, several arguments have been streamlined, sometimes with the help of Macaulay 2.
In Section 3 we present the structure of the Picard group of $\cM_g$ while in Section 4 we recall Harris and Mumford's spectacular application of the Grothendieck-Riemann-Roch theorem \cite{HM} in order to compute the canonical class $K_{\mm_g}$ and then discuss Pandharipande's recent lower bound on the slope of $\mm_g$. In Section 5 we present a much shorter proof of the Harris-Mumford Theorem \ref{hme} using syzygies of canonical
curves. Relying somewhat on Mukai's earlier work, we highlight the importance of the locus $\K_g\subset \cM_g$ of curves lying on $K3$ surfaces in order to construct effective divisors on $\mm_g$ having small slope and produce a criterion which each divisor of small slope must satisfy (Section 6). We then explain how to construct and compute the class of certain effective divisors on $\mm_g$ defined in terms of Koszul cohomology of line bundles on curves (cf. \cite{F2}, \cite{F3}). In Section 7 we prove that
$\mm_{22}$ is of general type.

\section{How rational is $\cM_g$?}

As a matter of terminology, if $\textbf{M}$ is a Deligne-Mumford stack, we denote
by $\cM$ its coarse moduli space. This is contrary to the convention set in \cite{ACV}
but for moduli spaces of curves it makes sense from a traditionalist point of view. Throughout the paper we denote by $\rem_g:\mbox{Sch}\rightarrow \mbox{Sets}$ the contravariant functor (stack) of stable curves of genus $g$,  which associates to every
scheme $S$ the set $\rem_g(S)$ of isomorphism classes
of relative stable curves $f:X\rightarrow S$ of genus $g$.

The functor $\rem_g$ is not representable, for this would imply that each iso-trivial family of stable
 curves is actually trivial. This, of course, is not the case. To remedy this problem one looks for a compromise solution by retaining the requirement that the
moduli space of curves be a scheme, but relaxing the condition that it represent $\rem_g$. The result is the
\emph{coarse moduli space of curves} $\mm_g$ which is an irreducible projective variety of dimension $3g-3$ with finite quotient singularities, cf. \cite{DM}, \cite{GIT}, \cite{M2}. For a family of stable curves $[f:X\rightarrow S]\in \rem_g(S)$
we shall denote by $m_f:S\rightarrow \mm_g$ the associated moduli map.

\subsection{Brill-Noether theory}\hfill

We recall a few basic facts from Brill-Noether theory, cf. \cite{ACGH}. For a smooth curve $C$ of genus
$g$ and for integers $d, r\geq 0$, one considers the cycle inside the Jacobian
$$W^r_d(C):=\{L\in \mbox{Pic}^d(C): h^0(C, L)\geq r+1\}.$$
The variety of linear series of type $\mathfrak g^r_d$ is defined as
$$G^r_d(C):=\{(L, V): L\in W^r_d(C), V\in \GG(r+1,
H^0(L))\}.$$ There is an obvious forgetful map $c:G^r_d(C)\rightarrow
W^r_d(C)$ given by $c(L, V):=L$.

We fix a point $l=(L, V)\in G^r_d(C)$, and describe the tangent
space $T_l(G^r_d(C))$. One has  the standard
identification $T_L(\mbox{Pic}^d(C))=H^1(C, \OO_C)=H^0(C, K_C)^{\vee}$ and
we denote by
$$\mu_0(L, V):V\otimes H^0(C, K_C\otimes L^{\vee})\rightarrow
H^0(C, K_C)$$ the \emph{Petri map} given by multiplication of sections.
The deformations of $[L]\in \mbox{Pic}^d(C)$ preserving the
space of sections $V$ correspond precisely to those elements $\phi \in H^0(C,
K_C)^{\vee}$ for which $\phi_{| \mathrm{Im }\ \mu_0(L, V)}=0$.
One obtains an exact sequence
$$0\longrightarrow \mbox{Hom}(V,
H^0(C, L)/V) \longrightarrow T_l(G^r_d(C))\longrightarrow \mbox{Ker }
\mu_0^{\vee}\longrightarrow 0.$$  It follows that $G^r_d(C)$ is smooth
and of dimension $$\rho(g, r, d):=g-(r+1)(g-d+r)$$ at the point $l$ if and only if
$\mu_0(L, V)$ is injective.

The Gieseker-Petri Theorem states that if $[C]\in \cM_g$ is general, then the
Petri map
$$\mu_0(L):H^0(C, L)\otimes H^0(C, K_C\otimes L^{\vee})\rightarrow
H^0(C, K_C)$$ is injective for \emph{every} $L\in \mathrm{Pic}^d(C)$. In particular
it implies that both $W^r_d(C)$ and $G^r_d(C)$ are irreducible
varieties of dimension $\rho(g, r, d)$. The variety $G^r_d(C)$ is smooth while  $\mbox{Sing } W^r_d(C)=W^{r+1}_d(C)$. Furthermore,
$W^r_d(C)=\emptyset$ if $\rho(g, r, d)<0$.

The first rigorous proof of Petri's theorem is due to Gieseker. The original proof has
been greatly simplified by Eisenbud and Harris, cf. \cite{EH4}, using degeneration to curves of compact type and the theory of limit
linear series. A very different proof, in which the degeneration argument is replaced by Hodge theory and the geometry of curves on $K3$ surfaces, has been found by Lazarsfeld \cite{La1}.

If $[C, p]\in \cM_{g, 1}$ and $l=(L, V)\in G^r_d(C)$, we define the \emph{vanishing sequence}
of $l$ at $p$
$$a^l(p):0\leq a_0^l(p)<\ldots <a_r^l(p)\leq d$$
by ordering the set $\{\mbox{ord}_p(\sigma)\}_{\sigma \in V}$. The \emph{ramification sequence}
of $l$ and $p$
$$\alpha^l(p): 0\leq \alpha_0^l(p)\leq \ldots  \leq \alpha_r(p)\leq d-r$$
is obtained from the vanishing sequence by setting $\alpha_i^l(p):=a_i^l(p)-i$ for $i=0\ldots r$.

The theory of degenerations of linear series (in the case of curves of compact type)
has been beautifully developed by Eisenbud and Harris \cite{EH1}. The major successes of the theory
include a simple proof of the Brill-Noether-Petri theorem  cf. \cite{EH4} and especially its essential use in the work
on the Kodaira dimension of $\mm_g$ cf. \cite{EH3}.
\begin{definition}
If $X$ is a stable
curve whose dual graph is a tree, a \emph{limit linear series} $\mathfrak
g^r_d$ on $X$, consists of a collection of linear series
$$l=\bigl\{l_Y=\bigl(L_Y, V_Y\subset H^0(L_Y)\bigr)\in G^r_d(Y): Y \mbox{ component of } X\bigr\}$$
satisfying the following compatibility condition: If $p\in Y\cap Z$ is a node
lying on two irreducible components $Y$ and $Z$ of $X$, then
$$a_i^{l_Y}(p)+a_{r-i}^{l_Z}(p)\geq d, \ \mbox{ for } i=0\ldots r.$$
\end{definition}
Limit linear series behave well in families: If $\textbf{M}_g^*\subset \rem_g$ denotes the open substack of
tree-like curves, then there exists an algebraic stack $\sigma:\widetilde{\mathfrak G}^r_d\rightarrow \textbf{M}_g^*$
classifying limit linear series. Each irreducible component of $\widetilde{\mathfrak G}^r_d$ has dimension at least
$3g-3+\rho(g, r, d)$. In particular if $l\in \overline{G}^r_d(C)=\sigma^{-1}(C)$ is a limit $\mathfrak g^r_d$ on a curve
$[C]\in \cM_g^*$ belonging to a component of $\overline{G}^r_d(C)$ of the expected dimension $\rho(g, r, d)$, then $l$ can be smoothed
to curves in an open set of $\cM_g$ (cf. \cite{EH1}).

\subsection{Severi's proof of the unirationality of $\cM_g$ when
$g\leq 10$}\hfill \vskip 2pt

We outline Severi's classical argument \cite{S} showing that $\cM_g$ is unirational for small genus (for a beautiful modern presentation see \cite{AC1}). The idea is very simple: One tries to represent the general curve $[C]\in \cM_g$ as a nodal plane curve $\Gamma \subset \PP^2$ of minimal degree $d$ such that $\rho(g, 2, d)\geq 0$ and then show that the nodes are in general position. Since the varieties of plane curves with \emph{fixed} nodes are linear spaces, hence rational varieties, this implies that
$\cM_g$ is unirational.

We fix  $d\geq (2g+8)/3$ and set $\delta:={d-1\choose 2}-g$. We consider the \emph{Severi variety}
$$U_{d, g}:=\{[\Gamma \hookrightarrow \PP^2]: \mathrm{deg}(\Gamma)=d,\ \Gamma \mbox{ is a nodal irreducible plane curve}, \ p_g(\Gamma)=g\}.$$ It is well-known that $U_{d, g}$ is an irreducible variety of dimension
$$\mathrm{dim }\ U_{d, g}=\mathrm{dim} \ \cM_g+ \rho(g, 2, d)+\mathrm{dim }\ PGL(3)=3d+g-1.$$
Furthermore, there exists a global desingularization map $\nu_{d, g}: U_{d, g}-->  \cM_g$ which associates
to each plane curve the class of its normalization. The Brill-Noether theorem guarantees that $\nu_{d, g}$ is surjective.
(Indeed, since $\rho(g, 2, d)\geq 0$ one has that $G^2_d(C)\neq \emptyset$ and it is straightforward to prove that a general
$\mathfrak g^2_d$ corresponds to a nodal model of a general curve $[C]\in \cM_g$, see for instance \cite{EH1}).

One defines the incidence correspondence between curves and their nodes
$$\Sigma:=\{\bigl([\Gamma \hookrightarrow \PP^2], p_1+\cdots +p_{\delta}\bigr)\in U_{d, g}\times \mathrm{Sym}^{\delta}(\PP^2): \{p_1, \ldots, p_{\delta}\}=\mathrm{Sing}(\Gamma)\},$$
together with the projection $\pi_2: \Sigma \rightarrow \mathrm{Sym}^{\delta}(\PP^2)$.
The fibres of $\pi_2$ being linear spaces, in order to conclude that $\Sigma$ is rational (and hence $\cM_g$  unirational), it suffices to
prove that $\pi_2$ is dominant.
A necessary condition for $\pi_2$ to be dominant is that
$$\mathrm{dim} \ \Sigma=3d+g-1\geq 2\delta.$$
This, together with the condition $\rho(g, 2, d)\geq 0$, implies that $g\leq 10$. We have the following result  \cite{S}, \cite{AC1} Theorem 3.2:
\begin{theorem}
We fix non-negative integers $g, d, \delta$ satisfying the inequalities
$$
\delta={d-1\choose 2}-g, \ \mbox{  }  \rho(g, 2, d)\geq 0 \mbox{ and  }  \ 3d+g-1\geq 2\delta.
$$
If $p_1, \ldots, p_{\delta}\in \PP^2$ are general points and $(n, \delta)\neq (6, 9)$, then there
exists an irreducible plane curve $\Gamma \subset \PP^2$ of degree $d$ having nodes at $p_1, \ldots, p_{\delta}$ and no other singularity.  It follows that $\cM_g$ is unirational for $g\leq 10$.
\end{theorem}

\begin{remark} As explained Severi's argument cannot be extended to any $\cM_g$ for $g\geq 11$. In a similar direction, a classical result of B. Segre \cite{Seg} shows that if $S$ is any algebraic surface and $\Xi\subset S\times V$ is any algebraic system of \emph{smooth} genus $g$ curves contained in $S$, then whenever $g>6$, the moduli map $m_{\Xi}: V-->\cM_g$ cannot be dominant, that is, no algebraic system of smooth curves of genus $g>6$ with general moduli can lie on any given surface.
\end{remark}

\subsection{Verra's proof of the unirationality of $\cM_g$ for
$11\leq g\leq 14$} $\hfill$

We fix an integer $g\geq 11$ and aim to prove the unirationality of
$\cM_g$ by showing that a suitable component of a Hilbert scheme
of curves
$$\mbox{Hilb}_C:=\{C\subset \PP^r: p_a(C)=g, \mbox{deg}(C)=d\},$$
where $\rho(g, r, d)\geq 0$ is unirational. The component
$\mbox{Hilb}_C$ must have the property that the forgetful rational
map $$m_C: \mbox{Hilb}_C-->\cM_g$$ is dominant (in particular, the general point of $\mbox{Hilb}_C$  corresponds
to a smooth curve $C\subset \PP^r$). To prove that
$\mbox{Hilb}_C$ is unirational we shall use an incidence
correspondence which relates $\mbox{Hilb}_C$ to another Hilbert
scheme of curves $\mbox{Hilb}_D$ parameterizing  curves $D\subset
\PP^r$ such that $$\mbox{deg}(D)=d', \  g(D)=g'\ \mbox{ and  } H^1(D,
\OO_D(1))=0$$ (thus $r=d'-g'$). If $[D\hookrightarrow \PP^r]\in \mbox{Hilb}_D$ is a smooth
curve with $H^1(D, \OO_D(1))=0$, then trivially
$H^1(D, N_{D/\PP^r})=0$, which implies that $\mbox{Hilb}_D$ is smooth at the point $[D\hookrightarrow \PP^r]$ and of
dimension $$h^0(D, N_{D/\PP^r})=\chi(D, N_{D/\PP^r})=(r+1)d'-(r-3)(g'-1)$$
(see e.g. \cite{Se2}).
Moreover, there exists an open subvariety  $\cU_D\subset \mbox{Hilb}_D$ parameterizing smooth non-special curves $D\subset
\PP^r$  such that (i) the moduli map $m_D:\cU_D\rightarrow \cM_{g'}$
is dominant, and (ii) the restriction maps $\mu_f:\mbox{Sym}^f H^0(D,
\OO_D(1))\rightarrow H^0(D, \OO_D(f))$ are of maximal rank for all
integers $f$.

The correspondence between $\mbox{Hilb}_C$ and $\cU_D$ is
given by liaison with respect to hypersurfaces of a fixed degree $f$,
that is, via the variety
$$\Sigma:=\{(D, V): [D\hookrightarrow \PP^r]\in \cU_D, \ \mbox{ } V \in \GG \bigl(r-1, H^0(\PP^r,
\mathcal{I}_{D/\PP^r}(f))\bigr)\}.$$ One has a projection map
$u:\Sigma\rightarrow \cU_D$ given by $u(D, V):=[D]$, and a residuation
map  $$\mathrm{res}:\Sigma\rightarrow \mbox{Hilb}_C, \ \mbox{
} \mathrm{ res}(D, V):=[C\hookrightarrow \PP^r],$$ where $C\subset \PP^r$ is the scheme residual
to $D$ in the base locus of the linear system $|V|$. The morphism
$u:\Sigma\rightarrow \cU_D$ has the structure of a Grassmann bundle
corresponding to the vector bundle $\F$ over $\cU_D$ with fibres
$\F(D)=H^0(\PP^r, \I_{D/\PP^r}(f))$, thus clearly $\Sigma$ is unirational
provided that $\mbox{Hilb}_D$ is unirational. Since $\cU_D$ parameterizes
non-special curves, proving its unirationality is equivalent to showing that
the universal Picard variety $\mathfrak{Pic}_{g'}^{d'}\rightarrow
\cM_{g'}$ is unirational.

In order to show that $m_C:\mbox{Hilb}_C-->\cM_g$ is dominant (and thus, that the general curve $[C]\in
 \cM_g$ is linked to a curve $[D\hookrightarrow \PP^r]\in \cU_D$), it suffices to
exhibit a single nodal complete intersection $$
C\cup D=X_1 \cap \ldots \cap X_{r-1}$$  with  $X_i\in |\OO_{\PP^r}(f)|$, such that both
$C$ and $D$ are smooth and the Petri map
$$\mu_0(C):
H^0(C, \OO_C(1))\otimes H^0(C, K_C\otimes \OO_C(-1))\rightarrow
H^0(C, K_C)$$ is injective. Indeed, it is well-known (see e.g. \cite{Se2}) that via Kodaira-Spencer
theory, the
differential $dm_{[C]}:T_{[C]}(\mbox{Hilb}_C)\rightarrow
T_{[C]}(\textbf{M}_g)$ is given by the coboundary map $H^0(C,
N_{C/\PP^r})\rightarrow H^1(C, T_C)$ obtained by taking cohomology
in
 the exact sequence which defines the normal bundle of $C$:
\begin{equation}\label{normal}
0\longrightarrow T_C\longrightarrow T_{\PP^r}\otimes \OO_C
\longrightarrow N_{C/\PP^r}\longrightarrow 0.
\end{equation}

On the other hand, one has the pull-back of the Euler sequence from
$\PP^r$
\begin{equation}\label{euler}
0\longrightarrow \OO_C\longrightarrow H^0(C, \OO_C(1))^{\vee}\otimes
\OO_C(1)\longrightarrow T_{\PP^r}\otimes \OO_C\longrightarrow 0,
\end{equation}
and after taking cohomology we identify $H^1(T_{\PP^r}\otimes
\OO_C)$  with the dual of the Petri map $\mu_0(C)$. Thus if
$\mu_0(C)$ is injective, then $m_C$ is a dominant map around $[C\hookrightarrow \PP^r]$.

The numerical invariants of $C$ and $D$ are related by well-known
formulas for linked subschemes of $\PP^r$, see \cite{Fu} Example 9.1.12: Suppose
$C\cup D=X_1\cap \ldots \cap X_{r-1}$ is a nodal complete intersection with $X_i\in
|\OO_{\PP^r}(f)|$ for $1\leq i\leq r-1$. Then one has that
\begin{equation}\label{link1}
\mbox{deg}(C)+\mbox{deg}(D)=f^{r-1},
\end{equation}
\begin{equation}\label{link2}
2(g(C)-g(D))=\bigl((r-1)f-r-1\bigr)(\mbox{deg}(C)-\mbox{deg}(D)),
\mbox{ and }
\end{equation}
\begin{equation}\label{link3}
\# (C\cap D)=\mbox{deg}(C)\cdot \bigl((r-1)f-r-1\bigr)+2-2g(C).
\end{equation}

We shall prove that if we choose
\begin{equation}\label{link4}
f=\frac{r+2}{r-2}\in \mathbb Z,
\end{equation}
the condition that $\OO_D(1)$ be non-special is equivalent to
 $h^0(\PP^r, \mathcal{I}_{C/\PP^r}(f))=r-1$. Furthermore, under the same assumption,
$\mu_0(C)$ is injective if and only if $\I_{D/\PP^r}(f)$ is
globally generated.

To summarize, we have reduced the problem of showing that $\cM_g$ is
unirational to showing (1) that the universal Picard variety
$\mathfrak{Pic}^{d'}_{g'}$ is unirational and (2) that one can find a non-special curve
$D\subset \PP^{d'-g'}$ whose ideal is cut out by hypersurfaces of degree $f$.
This program can be carried out provided one can solve the equations (\ref{link1}), (\ref{link2}),
(\ref{link3}) and (\ref{link4}) while keeping $\rho(g, r, d)\geq 0$.
To prove (1) Verra
relies on Mukai's work on the geometry of Fano $3$-folds and
on the existence of models of canonical curves of genus $g\leq 9$ as
linear sections of certain rational homogeneous varieties. We first
explain Mukai's work on existence of models of canonical curves of
genus $g\leq 9$. The standard references for this part are \cite{Mu1}, \cite{Mu2},
\cite{Mu3}:

\begin{theorem}\label{mukai}
We fix integers $g\leq 9, r\geq 3$ and $d\geq g+3$. Then the
universal Picard variety $\mathfrak Pic_g^d$ is unirational.
Moreover, if $\H_{d, g, r}$ denotes the unique component of the
Hilbert scheme of curves whose generic point corresponds to a smooth
curve $D\subset \PP^r$ with $\mathrm{deg}(D)=d$, \ $g(D)=d$ and
$H^1(D, \OO_D(1))=0$, then $\H_{d, g, r}$ is unirational as well.
\end{theorem}

The theorem is easily proved for $g\leq 6$ because, in this case, the
general canonical curve of genus $g$ is a complete intersection in
some projective space: For instance,
the canonical model of the general curve $[C]\in \cM_4$ is a $(2, 3)$ complete intersection in $\PP^4$ while
the canonical model of the general curve $[C]\in \cM_5$ is a $(2, 2, 2)$ complete intersection in $\PP^4$.
In the cases $g=7, 8, 9$,  Mukai
has found a  rational homogeneous space $\GG_g\subset \PP^{\mathrm{dim}(\GG_g)+g-2}$
such that $K_{\GG_g}=\OO_{\GG_g}(-\mathrm{dim}(\GG_g)+2)$, with the property that the general canonical curve
of genus $g$ appears as a curve section of $\GG_g$.

For $g=8$, we consider a vector space $V\cong \mathbb C^6$ and we introduce the Grassmannian of lines
$\GG_8:=\GG(2, V)$ together with the Pl\"ucker embedding
$\GG(2, V) \hookrightarrow \PP(\wedge^2 V)$. Then $K_{\GG_8}=\OO_{\GG_8}(-6)$. If $H\in
\GG(8, \wedge^2 V)$ is a general $7$-dimensional projective subspace and $C_H:=\GG_8 \cap \PP(H)\hookrightarrow \PP(H)$,
then by adjunction $K_C=\OO_C(1)$. In other words, a transversal codimension $7$
linear section of $\GG_8$ is a canonical curve of genus 8. Mukai
showed that any curve
$[C]\in \cM_8$ such that $W^1_4(C)=\emptyset$, can be recovered in
this way (cf. \cite{Mu1}).

The case $g=9$ is described in \cite{Mu3}: One takes $\GG_9:=\textbf{SpG}(3, 6)\subset \PP^{13}$
to be the \emph{symplectic Grassmannian}, that is, the Grassmannian
of Lagrangian subspaces of a $6$-dimensional symplectic vector
space $V$. Then $\mbox{dim}(\GG_9)=6$ and $K_{\GG_9}=\OO_{\GG_9}(-4)$. Codimension $5$ linear sections $\GG_9\cap H_1\cap \ldots \cap H_5\subset \PP^8$ are canonical curves of
genus $9$. A genus $9$ curve is a transversal
section of $\GG_9$ if and only if $W^1_5(C)=\emptyset$. In particular a general $[C]\in \cM_9$ is obtained through
this construction. Finally, we mention Mukai's construction for $g=7$, cf. \cite{Mu2}: For a vector space $V\cong \mathbb C^{10}$, the subset
of the Grassmannian $\GG(5, V)$ consisting of totally isotropic quotient spaces has two connected components, one of which is the $10$-dimensional
spinor variety $\GG_7\subset \PP^{15}$.

\noindent \emph{Proof of Theorem \ref{mukai}.}
It is enough to deal with the cases $g=7, 8,  9$. For each
integer $d\geq g+3$, we fix non-zero integers $n_1, \ldots , n_g$
such that
$$2g-2+n_1+\cdots+n_g=d$$ and note that for every $[C]\in \cM_g$,
the map $C^g\rightarrow \mbox{Pic}^d(C)$ sending $$(x_1, \ldots,
x_g)\mapsto K_C\otimes \OO_C(n_1 x_1+\cdots +n_g\ x_g)$$ is
surjective.  Then the  rational map
$\phi:\GG_g^g-->\mathfrak{Pic}_g^d$  defined by
$$\phi(x_1, \ldots, x_g):=\bigr(C_x=\GG_g \cap \PP(\langle x_1, \ldots, x_g\rangle ),
\ K_{C_x}\otimes \OO_{C_x}(n_1\ x_1+\cdots+n_g\
x_g)\bigr),$$ is dominant. Thus $\mathfrak{Pic}_g^d$ is unirational.

To establish the unirationality of $\H_{d, g, r}$ when $3\leq r\leq
d-g$, we consider the dominant map $f:\H_{d, g, r}\rightarrow
\mathfrak{Pic}_{g}^d$ given by $f([C\subset \PP^r]):=[C, \OO_C(1)]$.
The fibres of $f$ are obviously rational varieties. It follows that
$\H_{d, g, r}$ is unirational too.
\hfill $\Box$

Now we explain Verra's work cf. \cite{Ve}, focusing on the cases
$g=11, 14$. Several (admittedly beautiful) arguments of classical geometric nature have been straightened
or replaced by Macaulay 2 calculations in the spirit of \cite{F2}, Theorems 2.7, 2.10 or \cite{ST}.

\begin{theorem}\label{verra}
The moduli space of curves $\mm_g$ is unirational for $11\leq
g\leq 14$.
\end{theorem}
\noindent \emph{Proof for $g=11, 14$.}
We place ourselves in the situation when $f=(r+2)/(r-2)\in \mathbb Z$.
The relevance of this condition is that a surface complete intersection
of type $(f, \ldots, f)$ in $\PP^r$ is a canonical surface in $\PP^r$.
We consider a nodal complete intersection $$C\cup D=X_1\cap
\ldots \cap X_{r-1}$$ with $X_i\in |\OO_{\PP^r}(f)|$, with $C$ and $D$
being smooth curves and with $g(C)=g$. Assuming that $\I_{C\cup D/\PP^r}(f)$ is
globally generated (this will be the case in all the situations we
consider),  then $C\cup D$ lies on a smooth complete
intersection of $r-2$ hypersurfaces of degree $f$, say  $S:=X_1\cap
\ldots \cap X_{r-2}$. Thus $S$ is a surface with $K_S=\OO_S(1)$ and
moreover $h^0(\PP^r, \I_{S/\PP^r}(f))=r-2$ (use the Koszul resolution of
$\I_{S/\PP^r}$). From the exact sequence
$$0\longrightarrow \I_{C/\PP^r}\longrightarrow
\I_{S/\PP^r}\longrightarrow \OO_S(-C)\longrightarrow 0,$$ we find
that
$h^0(\PP^r, \I_{C/\PP^r}(f))=h^0(S, \OO_S(C))+h^0(\PP^r, \I_{S/\PP^r}(f))=h^0(\OO_S(D))+
r-2$ \  (Note that $C+D\in |\OO_S(f)|$). Moreover, from the exact
sequence
$$0\longrightarrow \OO_S(1)\otimes \OO_S(-D)\longrightarrow
\OO_S(1)\longrightarrow \OO_D(1)\longrightarrow 0,$$ using also Serre
duality, we obtain that
$$h^0(S, \OO_S(D))=h^2(S, \OO_S(H-D))-h^2(S, \OO_S(H))=1+h^1(D, \OO_D(1)).$$
Therefore $\OO_D(1)$ is non-special if and only if
\begin{equation}\label{nrquad}
h^0(\PP^r, \I_{C/\PP^r}(f))=r-1.
\end{equation}
Assume now that $r=d'-g'$ and that $g(D)=g'$, $\mbox{deg}(D)=d'$
 $g(C)=g$ and  $\mbox{deg}(C)=d$, where these invariants are related
by the formulas (\ref{link1})-(\ref{link3}). Using a simple argument
involving diagram chasing, we claim that the Petri map
$$\mu_0(C): H^0(C, \OO_C(1))\otimes H^0(C, K_C(-1))\rightarrow H^0(C, K_C)$$
is of maximal rank if and only if the multiplication map
\begin{equation}\label{formf}
\nu_D(f): H^0(\PP^r, \I_{D/\PP^r}(f))\otimes H^0(D, \OO_D(1))\rightarrow
H^0(\PP^r, \I_{D/\PP^r}(f+1))
\end{equation}
is of maximal rank (see  \cite{Ve}, Lemma 4.4). Indeed, since $K_C(-1)=\OO_C(C)$,
we find that
$$\mathrm{Ker}\bigl(\mu_0(C)\bigr)=\mathrm{Ker}\{\mu_S: H^0(S, \OO_S(C))\otimes H^0(S, \OO_S(C+D))\rightarrow H^0(S, \OO_S(D+2C))\}.$$

Next we note that $\I_{D/S}(f)=\OO_S(C)$ and then the claim follows by applying the Snake Lemma to the
diagram obtained by taking cohomology in the sequence
$$0\longrightarrow H^0(\OO_{\PP^r}(1))\otimes \I_{S/\PP^r}(f)\longrightarrow H^0(\OO_{\PP^r}(1))\otimes \I_{D/\PP^r}(f)\longrightarrow H^0(\OO_{\PP^r}(1))\otimes \OO_S(C)\longrightarrow 0.$$

Thus to prove that the moduli map $m_C:\mathrm{Hilb}_C-->\cM_g$ is dominant, it suffices to exhibit a smooth curve $[D]\in
\mbox{Hilb}_D$ such that the map
$\nu_D(f)$ is injective.

Having explained this general strategy, we start with the case
$g=14$ and suppose that $[C]\in \cM_{14}$ is a curve satisfying Petri's
theorem, hence $\mbox{dim } W^1_8(C)=\rho(14, 1, 8)=0$.  For each pencil $A\in W^1_8$ we have that $L:=K_C\otimes A^{\vee}\in
W_{18}^6(C)$ and when $[C]\in \cM_{14}$ is sufficiently general, each such linear series gives rise to
 an embedding
$C\stackrel{|L|}\hookrightarrow \PP^6$. By Riemann-Roch we obtain
that
$$\mbox{dim }\mbox{Ker}\{ \mbox{Sym}^2 H^0(C, L)\rightarrow H^0(C,
L^{\otimes 2})\}={8\choose 2}-\bigl(2\ \mbox{deg}(C)+1-g(C)\bigr)=5,$$ that is
$C$ lies on precisely  $5$ independent quadrics $Q_1, \ldots, Q_5\in |\OO_{\PP^6}(2)|$.
Writing
$$Q_1\cap \ldots \cap Q_5=C\cup D,
$$
we find that
$g(D)=8$ and $\mbox{deg}(D)=14$. In particular, we  also have
that $H^1(D, \OO_D(1))=0$. Thus we have reduced the problem of
showing that $\cM_{14}$ is unirational to two questions:

\noindent (1)  $\mathfrak{Pic}_8^{14}$ is unirational. This  has
already been achieved (cf. Theorem \ref{mukai}).

\noindent (2) If $D\subset \PP^6$ is a
general smooth curve with $\mbox{deg}(D)=14$ and $g(D)=8$, then the
map $$\nu_D(2): H^0(\PP^6, \I_{D/\PP^6}(2))\otimes H^0(D, \OO_D(1))\rightarrow
H^0(\PP^6, \I_{D/\PP^6}(3))$$ is an isomorphism. This is proved using
liaison and a few classical arguments (cf. \cite{Ve}, Propositions 5.5-5.16). We shall present a slightly more
direct proof using Macaulay2.

 When $g=11$, we choose $d=14$
and $r=4$, hence $f=3$. We find that if
$[C]\in \cM_{11}$ is general then $\mbox{dim } W^4_{14}(C)=\rho(11, 4, 14)=6$ and  $h^1(C, L)=1$ for every $L\in
W^4_{14}(C)$. Moreover, for a general linear series $L\in W^4_{11}(C)$,
$$\mbox{dim }\mbox{Ker}\{ \mbox{Sym}^3 H^0(C, L)\rightarrow H^0(C,
L^{\otimes 3})\}=3,$$
(in particular condition (\ref{nrquad}) is satisfied). Hence there are
hypersurfaces $X_1, X_2, X_3\in |\OO_{\PP^4}(3)|$ such that $X_1\cap
X_2 \cap X_3=C\cup D$. Moreover, $g(D)=9$ and $\mbox{deg}(D)=13$,
and the unirationality of $\mm_{11}$ has been reduced to showing that:

\noindent (1) $\mathfrak{Pic}_9^{13}$ is unirational. This again follows from
Theorem \ref{mukai}.

\noindent (2) If $D\subset \PP^4$ is a general smooth curve with $\mbox{deg}(D)=13$ and $g(D)=9$, then the map
$$\nu_D(3): H^0(\PP^4, \I_{D/\PP^4}(3))\otimes H^0(\OO_D(1))\rightarrow
H^0(\PP^4, \I_{D/\PP^4}(4))$$
is injective.
\hfill
$\Box$

We complete the proof of Theorem \ref{verra}, and we focus on the case $g=14$. A similar argument deals with the case $g=11$:

\begin{theorem} If $D\stackrel{|L|}\hookrightarrow \PP^6$ is the embedding corresponding to a
general curve $[D, L]\in \mathfrak{Pic}_{8}^{14}$, then the multiplication map
$$H^0(\PP^6, \I_{D/\PP^6}(2))\otimes H^0(\PP^6, \OO_{\PP^6}(1))\rightarrow H^0(\PP^6, \I_{D/\PP^6}(3))$$
is an isomorphism.
\end{theorem}
\begin{proof}
We consider $11$ general points in $\PP^2$ denoted by $p_1, \ldots, p_5$ and $q_1,\ldots, q_6$ respectively, and define the
linear system
$$H\equiv 6h-2(E_{p_1}+\cdots + E_{p_5})-(E_{q_1}+\cdots + E_{q_6})$$
 on the blow-up $S=\mbox{Bl}_{11}(\PP^2)$. Here $h$ denotes the
pullback of the line class from $\PP^2$. Using the program
Macaulay2 it is easy to check that $S\stackrel{|H|}\hookrightarrow
\PP^6$ is an embedding and the graded Betti diagram of $S$ is the following:
\[
\begin{matrix}
 1 & -  & -  & - & -  \\
 - & 5 & - & - &- \\
 - & -  & 15  & 16& 15
\end{matrix}
\]
Thus $S$ satisfies property $(N_1)$.  To carry out this calculation we
chose the $11$ points in $\PP^2$ randomly using the Hilbert-Burch
theorem so that they satisfy the Minimal Resolution Conjecture
(see \cite{ST} for details on how to pick random points in
$\PP^2$ using Macaulay). Next we consider a curve $D\subset S$ in
the linear system
 \begin{equation} \label{c16}
D\equiv 10h-3(E_{p_1}+E_{p_2})-4\sum_{i=3}^5 E_{p_i}-E_{q_1}-E_{q_2}-
2\sum_{j=3}^6 E_{q_j}. \end{equation} By using Macaulay2, we
pick $D$ randomly in its linear system and then check that $D$ is
smooth, $g(D)=8$ and $\mbox{deg}(D)=14$. We can compute directly the Betti diagram
of $D$:
\[
\begin{matrix}
 1 & -  & -  & - & -& \\
 - & 7 & - & - &- &\\
 - & -  & 35  & 56& 35
\end{matrix}
\]
Hence $K_{1, 1}(D, \OO_D(1))=0$, which shows that $\nu_2(D)$ is an isomorphism. This last part
 also follows directly: Since $S$ is cut out by quadrics, to
conclude that $D$ is also cut out by quadrics, it suffices to show that the
map
$$\nu_S: H^0(S, \OO_S(H))\otimes H^0(S, \OO_S(2H-D))\rightarrow H^0(S, \OO_S(3H-D))$$ is
surjective (or equivalently injective). Since $h^0(S,\OO_S(2H-D))=2$,
from the base point free pencil trick we get that
$\mbox{Ker}(\nu(S))=H^0(S, \OO_S(D-H))=0$, because $D-H$ is clearly not effective for a general choice of the $11$
points in $\PP^2$.
\end{proof}

We end this section, by pointing out that already existing results in \cite{CR3}, coupled with recent advances in higher dimensional birational geometry, imply the following:
\begin{theorem}\label{16}(Chang-Ran)
The moduli space $\mm_{16}$ is a uniruled variety.
\end{theorem}
\begin{proof} Chang and Ran proved in \cite{CR3} that $\kappa(\mm_{16})=-\infty$, by exhibiting an explicit collection of curves $\{F_i\}_{i=1}^n\subset \mm_{16}$,
with the property that each $F_i$ lies on a divisor $D_i\subset \mm_{16}$ such that $F_i$ is nef as a curve on $D_i$ with respect to $\mathbb Q$-Cartier divisors, and moreover $$F_i\cdot \sum_{j=1}^n D_j>0\ \mbox{ for }\ i=1, \ldots, n.$$
By explicit calculation they noted that $F_i\cdot K_{\mm_{16}}<0$ for  $i=1, \ldots, n$. This clearly implies that $K_{\mm_{16}}$ is not pseudo-effective. Since pseudo-effectiveness of the canonical bundle is a birational property, the canonical bundle of any smooth model of $\mm_{16}$ will lie outside the pseudo-effective cone as well. One can apply the the main result of \cite{BDPP} to conclude that $\mm_{16}$ is uniruled.
\end{proof}

\section{The Picard group of the moduli stack $\rem_g$}

For a stable curve $[C]\in \mm_g$ one can consider its \emph{dual
graph} with  vertices corresponding to the irreducible
components of $C$ and edges corresponding to nodes joining two
components. By specifying the dual graph, one obtains the topological stratification of $\mm_{g}$,
where the codimension $a$ strata correspond to the irreducible
components of the closure of the locus of curves $[C]$ having precisely $a$ nodes. The closure of the codimension $1$ strata
are precisely the boundary divisors of $\mm_{g}$: For $1\leq i\leq [g/2]$  we denote by $\Delta_{i}\subset
\mm_g$ the closure of the locus of stable curves $[C_1\cup C_2]$, where $C_1$ and $C_2$ are smooth curves of genera $i$ and $g-i$ respectively. Similarly, $\Delta_0\subset \mm_{g}$ is the closure of the locus of irreducible $1$-nodal stable curves. We have the decomposition
$$\mm_g=\cM_g\cup \Delta_0\cup \ldots \cup \Delta_{[g/2]}.$$

Next we describe the Picard group of the moduli stack
$\rem_g$. The difference between the Picard group of
the stack $\rem_g$ and that of the coarse moduli space $\mm_g$, while subtle, is not
tremendously important in describing the birational geometry of $\mm_g$. Remarkably, one can define $\mbox{Pic}(\rem_g)$
without knowing exactly what a stack itself is! This approach at least
respects the historical truth: In 1965
Mumford \cite{M1} introduced the notion of a sheaf on the functor (stack)
$\rem_g$. One had to wait until 1969 for the definition of a
Deligne-Mumford stack, cf. \cite{DM}.

\begin{definition}\label{picst}
A sheaf
$\mathcal{L}$ on the stack $\rem_g$ is an assignment of a sheaf $\mathcal{L}(f)$ on $S$ for every family $[f:X\rightarrow S] \in
 \rem_g(S)$, such that for any morphism of schemes
 $\phi:T\rightarrow S$, if $p_2:X_T:=X\times_S T\rightarrow T$ denotes
 the family obtained by pulling-back $f$, then there is an isomorphism of sheaves
 over $T$ denoted by
 $$\mathcal{L}(\phi, f): \mathcal{L}(p_2)\rightarrow
 \phi^*(\mathcal{L}(f)).$$ These isomorphisms should commute
 with composition of morphisms between the bases of the families. Precisely,
  if $\chi: W\rightarrow T$ is another morphism and
 $$\sigma_2: X_W:=X_T\times_T W\rightarrow W\in \rem_g(W),$$ then $\mathcal{L}(\phi \chi, f)=\chi^*\mathcal
{L}(\phi, f)\circ \mathcal{L}(\chi, p_2)$.
If $\L$ and $\E$ are sheaves on $\rem_g$, we define their tensor product by setting
$$(\L\otimes \E)(f):=\L(f)\otimes \E(f)$$ for each $[f:X\rightarrow S]\in \rem_g(S)$.

A sheaf $\L$ on $\rem_g$ is a \emph{line bundle} if $\L(f)\in \mbox{Pic}(S)$
 for every $[f:X\rightarrow S]\in \rem_g(S)$. We denote by $\mbox{Pic}(\rem_g)$ the group
 of isomorphism classes of line bundles on $\rem_g$. 
 
 Similarly, for $i\geq 0$, one defines a codimension $i$ \emph{cycle class}  $\gamma \in A^i(\rem_g)$, to be a collection of assignments $\gamma(f)\in A^i(S)$
 for all $[f:X\rightarrow S]\in \rem_g(S)$, satisfying an obvious compatibility condition like in Defintion \ref{picst}
\end{definition}
\begin{example} For each $n\geq 1$ the \emph{ Hodge classes} $\lambda_1^{(n)}\in \mbox{Pic}(\rem_g)$ are defined by taking $\lambda_1^{(n)}(f):=c_1(\mathbb E_n(f))$, where the assignment $$\mm_g(S)\ni [f:X\rightarrow S] \mapsto \mathbb E_n(f):= f_*(\omega_f^{\otimes n}),$$ gives rise to a vector bundle $\mathbb E_n$ on $\rem_g$ for each $n\geq 1$. Clearly $\mbox{rank}(\mathbb E_1)=g$ and $\mbox{rank}(\mathbb E_n)=(2n-1)(g-1)$ for $n\geq 2$. One usually writes $\mathbb E:=\mathbb E_1$.
Similarly, one can define the higher Hodge classes $\lambda_i^{(n)}\in A^i(\rem_g)$, by taking $$\lambda_i^{(n)}(f):=c_i(\mathbb E_n(f))\in A^i(S).$$
 It is customary to write that $\lambda_i:=\lambda_i^{(1)}$ and sometimes, $\lambda:=\lambda_1$.
\end{example}

There is an obvious group homomorphism $\rho:\mbox{Pic}(\mm_g)\rightarrow \mbox{Pic}(\rem_g)$
defined by $\rho(\mathcal{L})(f):=m_f^*(\mathcal{L})$ for every $\mathcal{L}\in \mbox{Pic}
(\mm_g)$ and $[f:X\rightarrow S]\in \rem_g(S)$.

To get to grips with the group $\mbox{Pic}(\rem_g)$ one can also use the GIT realization of the moduli space
and consider for each $\nu\geq 3$ the Hilbert scheme
$\bf{Hilb}_{g, \nu}$ of $\nu$-canonical stable embedded curves $C\subset \PP^{(2\nu-1)(g-1)-1}$. One has an isomorphism of varieties cf. \cite{GIT}, \cite{M2}
$$\mm_g\cong {\bf{Hilb}_{g, \nu}}//PGL\bigl((2\nu-1)(g-1)\bigr).$$
Using this  we can define an isomorphism of groups
$$\beta: \mbox{Pic}(\rem_g)\rightarrow
\mbox{Pic}({\bf{Hilb}_{g, \nu}})^{PGL\bigl((2\nu-1)(g-1)\bigr)}.$$
If $\sigma: \mathcal{C}_{g, \nu}\rightarrow {\bf{Hilb}_{g, \nu}}$ denotes the universal $\nu$-canonically embedded curve, where we have that
$\mathcal{C}_{g, \nu}\subset {\bf{Hilb}_{g, \nu}}\times \PP^{(2\nu-1)(g-1)-1}$,
we set $\beta(\mathcal{L}):=\mathcal{L}(\sigma)\in \mathrm{Pic}({\bf{Hilb}_{g, \nu}}).$

To define $\beta^{-1}$ we start with a line bundle $L\in \mathrm{Pic}({\bf{Hilb}}_{g, \nu})$ together with a
fixed lifting of the $PGL((2\nu-1)(g-1))$-action on ${\bf{Hilb}_{g, \nu}}$ to $L$. For a family of stable curves $f:X\rightarrow S$, we choose
a local trivialization of the projective bundle $\PP\bigl(f_*(\omega_f^{\otimes \nu})\bigr)$, that is, we fix isomorphisms over $S_{\alpha}$
$$\PP\Bigl((f_{\alpha})_*\bigl(\omega_{f_{\alpha}}^{\otimes \nu}\bigr)\Bigr)\cong \PP^{(2\nu-1)(g-1)-1}\times S_{\alpha},$$
where $\{S_{\alpha}\}_{\alpha}$ is a cover of $S$ and $f_{\alpha}=f_{| f^{-1}(S_{\alpha})}:X_{\alpha}\rightarrow S_{\alpha}$. Since the Hilbert
scheme is a fine moduli space, these trivializations induce morphisms $g_{\alpha}: S_{\alpha}\rightarrow
{\bf{Hilb}_{g, \nu}}$ such that on $S_{\alpha}\cap S_{\beta}$ the morphisms $g_{\alpha}$ and $g_{\beta}$
differ by an element from $PGL\bigl((2\nu-1)(g-1)\bigr)$. The choice of the $L$-linearization ensures that the sheaves $\{g_{\alpha}^*(L)\}_{\alpha}$ can be
glued to form a sheaf which we call $\beta^{-1}(L)(f)\in \mathrm{Pic}(S)$.

\begin{example} If $\OO_{{\bf{Hilb_{g, \nu}}}}(\delta)=\otimes_{i=0}^{[g/2]} \OO_{{\bf{Hilb_{g, \nu}}}}(\delta_i)$ is the divisor of all singular nodal curves on the universal curve
$\sigma:\cC_{g, \nu}\rightarrow {\bf{Hilb}_{g, \nu}}$, then
$$\rho([\Delta_0])=\beta^{-1}(\delta_0), \ \rho([\Delta_1])=2\beta^{-1}(\delta_1), \ \rho([\Delta_i])=\beta^{-1}(\delta_i)\ \mbox{ for }
2\leq i\leq [g/2].$$
 To put it briefly, we write that $\delta_i:=[\Delta_i]$ for $i\neq 1$ and $\delta_1:=\frac{1}{2}[\Delta_1]$
 in $\mathrm{Pic}(\rem_g)_{\mathbb Q}$.
 \end{example}

\begin{theorem}\label{pic}
1) The group homomorphism $\rho: \mathrm{Pic}(\mm_g)\rightarrow
\mathrm{Pic}(\rem_g)$ is injective with torsion cokernel. Thus
$$\rho_{\mathbb Q}:\mathrm{Pic}(\mm_g)_{\mathbb Q}\cong \mathrm{Pic}(\rem_g)_{\mathbb Q}.$$
2) For $g\geq 3$, the group $\mathrm{Pic}(\rem_g)$ is freely generated by the classes
$\lambda$, $\delta_0, \ldots,
\delta_{[g/2]}$.
\end{theorem}

From now on we shall identify $\mbox{Pic}(\rem_g)_{\mathbb Q}=\mbox{Pic}(\mm_g)_{\mathbb Q}$.
The first part of Theorem \ref{pic} was established by Mumford in \cite{M2} Lemma 5.8. The second part
is due to Arbarello and Cornalba \cite{AC2} and uses in an essential way Harer's theorem that $H^2(\cM_g, \mathbb Q)\cong \mathbb Q$. Unfortunately there is no purely algebraic proof of Harer's result yet.

\section{The canonical class of $\mm_g$}

In this section we explain the calculation of the canonical class of
$\mm_g$ in terms of the generators of $\mbox{Pic}(\rem_g)$, cf.
\cite{HM}. This calculation has been one of the spectacular successes of the Grothendieck-Riemann-Roch theorem.
In order to apply GRR one needs however a good modular interpretation of the cotangent bundle $\Omega_{\rem_g}^1$.
This is provided by Kodaira-Spencer theory. We first compute the
canonical class of the stack $\rem_g$, then we use the branched
cover $\rem_g \rightarrow \mm_g$ to obtain the canonical class of
the coarse moduli scheme $\mm_g$.

For every stable curve $[C]\in \mm_g$ we denote by $\Omega_C^1$ the sheaf of
K\"ahler differentials and by $\omega_C$ the locally free dualizing sheaf (see
\cite{Ba} for a down-to-earth introduction to the deformation theory of stable curves).
These sheaves sit in an exact sequence:
$$0\longrightarrow \mathrm{Torsion}(\Omega_C^1)\longrightarrow \Omega_C^1\longrightarrow \omega_C\stackrel{res}\longrightarrow
\bigoplus_{p\in \mathrm{Sing}(C)} \mathbb C_p\longrightarrow 0.$$
Kodaira-Spencer theory coupled with  Serre duality provides an
identification
$$T_{[C]}(\rem_g)=\mathrm{Ext}^1(\Omega_C^1, \OO_C)=H^0(C,
\omega_C\otimes \Omega_C^1)^{\vee}.$$ One can globalize this
observation and describe the cotangent bundle of $\rem_g$ as
follows. We denote by $\pi:\rem_{g, 1}\rightarrow \rem_g$ the
universal curve and we denote by $\omega_{\pi}$ the relative
dualizing sheaf and by $\Omega_{\pi}^1$ the sheaf of relative K\"ahler
differentials, respectively. Then by Kodaira-Spencer theory
we have the identification
$\Omega_{\rem_g}^1=\pi_*(\Omega_{\pi}^1\otimes \omega_{\pi})$ and
call the class $K_{\rem_g}=c_1(\Omega_{\rem_g}^1)\in
\mathrm{Pic}(\rem_g)$ the canonical class of the moduli stack
$\rem_g$. To compute the first Chern class of this push-forward
bundle we use the Grothendieck-Riemann-Roch theorem.

Suppose that we are given a proper map $f:X\rightarrow Y$ with
smooth base $Y$ and a sheaf $\F$ on $X$.  Then the
Grothendieck-Riemann-Roch (GRR) theorem reads
$$\mathrm{ch}\bigl(f_{!}(\F)\bigr)=f_*\bigl(\mathrm{ch}(\F)\cdot
\mathrm{td}(\Omega_f^1)\bigr)\in A^*(Y),\  \ \mbox{ where }$$
$$\mathrm{td}(\Omega_{f}^1):=1-\frac{c_1(\Omega_{f}^1)}{2}+\frac{c_1(\Omega_{f}^1)^2+c_2(\Omega_f^1)}{2}+
(\mbox{higher order terms})\ $$  denotes the Todd class.

\begin{remark}\label{higherhodge} One uses the GRR theorem to prove
Mumford's relation $$\kappa_1:=\pi_*(c_1(\omega_{\pi}^2))=12\lambda-\delta\in \mathrm{Pic}(\rem_g),$$
where $\delta:=\delta_0+\cdots+\delta_{[g/2]}$ is the total boundary (cf. \cite{M1} pg.
101-103). Similarly, for $n\geq 2$ we have the relation (to be used in Section 5), cf. \cite{M2} Theorem 5.10:
$$\lambda_1^{(n)}=\lambda+{n\choose 2}\kappa_1\in \mathrm{Pic}(\rem_g).$$
\end{remark}

To compute $K_{\rem_g}$ we set $f=\pi: \rem_{g, 1}\rightarrow
\rem_g$, $\F=\Omega_{\pi}^1\otimes \omega_{\pi}$, hence
$\pi_*\F=\Omega_{\rem_g}^1$ and $R^i \pi_*\F=0$ for $i\geq 1$.
Using Grothendieck-Riemann-Roch we can write:
\begin{equation}\label{can}
K_{\rem_g}=\pi_*\Bigl[\Bigl(1+c_1(\F)+\frac{c_1^2(\F)-2c_2(F)}{2}+\cdots\Bigr)\cdot
\Bigl(1-\frac{c_1(\Omega_{\pi}^1)}{2}+\frac{c_1(\Omega_{\pi}^1)^2+c_2(\Omega_{\pi}^1)}{12}+\cdots\Bigr)\Bigr]_1.
\end{equation}
Next we determine the Chern classes of $\F$. Suppose
$[f:X\rightarrow S]\in \rem_g(S)$ is a family of stable curves
such that both $X$ and $S$ are smooth projective varieties. Then
$\mathrm{codim}(\mathrm{Sing}(f), X)=2$ and the sheaf homomorphism
$\Omega_{f}^1\rightarrow \omega_f$ induces an isomorphism
$\Omega_f^1=\omega_f\otimes \mathcal{I}_{\mathrm{Sing}(f)}$ (in
particular, $\Omega_f^1$ is not locally free). This claim follows
from a local analysis around each point $p\in \mathrm{Sing}(f)$.
Since the versal deformation space of a node is $1$-dimensional, there exist affine coordinates $x, y$ on the fibres of
$f$ and an affine coordinate $t$ on $S$, such that locally around
$p$, the variety $X$ is given by the equation $xy=t^n$ for some
integer $n\geq 1$. By direct calculation in a neighbourhood
of $p$,
$$\Omega_f^1=\bigl(f^*\OO_C\cdot dx+f^*\OO_C \cdot
dy\bigr)/\bigl(x dy+y dx)\cdot \OO_C.$$ Similarly, the dualizing
sheaf $\omega_f$ is the free $\OO_X$ module generated by the
meromorphic differential $\eta$ given by $dx/x$ on the locus $x\neq
0$ and by $-dy/y$ on the locus $y\neq 0$, hence we find that locally
$\Omega_f^1=\mathcal{I}_{x=y=t=0}\cdot \omega_f$, which proves our
claim.

The sheaves $\omega_f$ and $\Omega_f^1$ agree in codimension $1$,
thus $c_1(\Omega_f^1)=c_1(\omega_f)$. An application of
Grothendieck-Riemann-Roch for the inclusion
$\mathrm{Sing}(f)\hookrightarrow X$, shows that
$c_2(\Omega_f^1)=[\mathrm{Sing}(f)]$. Then by the Whitney formula we
obtain that $c_1(\F)=2c_1(\omega_f)$ and
$c_2(\F)=[\mathrm{Sing}(f)]$. Since this analysis holds for an
arbitrary family of stable curves, the same relation must hold
for the universal curve over $\rem_g$. Returning to (\ref{can}),
we find the following formula:
$$K_{\rem_g}=\frac{13}{12}\pi_*\bigl(c_1(\omega_{\pi})^2)-\frac{11}{12}
\pi_*\bigl[\mathrm{Sing}(\pi)\bigr]=\frac{13}{12}\kappa_1-\frac{11}{12}\delta=13\lambda-2\delta\in
\mathrm{Pic}(\rem_g).$$

\begin{theorem}\label{canonicalcoarse}
For $g\geq 4$, the canonical class of the coarse moduli space $\mm_g$ is given by the formula
$$K_{\mm_g}\equiv 13\lambda-2\delta_0-3\delta_1-2\delta_2-\cdots-2\delta_{[g/2]}\in \mathrm{Pic}(\rem_g).$$
\end{theorem}
\begin{proof} We consider the morphism $\epsilon:\rem_g\rightarrow \mm_g$ which is simply branched along the divisor is the divisor
$\Delta_1\subset \mm_g$, hence $\epsilon^*([\Delta_1])=2\delta_1\in \mathrm{Pic}(\rem_g)$. The Riemann-Hurwitz formula gives that $K_{\rem_g}=\epsilon^*(K_{\mm_g})+\delta_1$ which finishes the proof.
\end{proof}
\begin{remark} A slight difference occurs in the case $g=3$. The morphism $\epsilon:\rem_3\rightarrow \mm_3$ is simply branched along both the boundary  $\Delta_1$ and the closure  of the hyperelliptic locus $$\cM_{3, 2}^1:=\{[C]\in \cM_3: W^1_2(C)\neq \emptyset\}.$$ It follows that
$K_{\rem_3}=\epsilon^*K_{\mm_3}+\delta_1+\rho ([\mm_{3, 2}^1])$,
hence $K_{\mm_3}=4\lambda-\delta_0$.
\end{remark}

Using Theorem \ref{canonicalcoarse}, we reformulate the problem of determining the Kodaira dimension of $\mm_g$
in terms of effective divisors: A sufficient condition for $\mm_g$ to be of general type is the existence of an effective divisor
$$D\equiv a\lambda-b_0\delta_0-\cdots-b_{[g/2]}\delta_{[g/2]}\in \mbox{Pic}(\mm_g),$$ with coefficients satisfying the following inequalities
\begin{equation}\label{inequ}
\frac{a}{b_0}<\frac{13}{2}, \ \mbox{  }\  \frac{a}{b_1}\leq \frac{13}{3} \ \mbox{ and }\  \ \frac{a}{b_i}\leq \frac{13}{2} \ \mbox{ for } \ 2\leq i\leq [g/2].
\end{equation}

This formulation using (\ref{inequ}) clearly suggests the definition of the following numerical invariant of the moduli space cf. \cite{HMo}:
If $\delta:=\delta_0+\cdots +\delta_{[g/2]}$ is the class of the total
boundary and $\mbox{Eff}(\mm_g)\subset \mbox{Pic}(\mm_g)_{\mathbb R}$ denotes the
 cone of effective divisors, then we can
define the \emph{slope function} $s:\mbox{Eff}(\mm_g) \rightarrow
\mathbb R\cup \{\infty\}$ by the formula
$$s(D):= \mbox{inf }\{\frac{a}{b}:a,b>0 \mbox{ such that }
a\lambda-b\delta-D\equiv \sum_{j=0}^{[g/2]} c_j\delta_j,\mbox{ where
}c_j\geq 0\}.$$ From the definition it follows that $s(D)=\infty$
unless $D\equiv a\lambda-\sum_{j=0}^{[g/2]} b_j\delta_j$ with
$a,b_j\geq 0$ for all $j$.  It is also well-known that $s(D)<\infty$
for any $D$ which is the closure of an effective divisor on
$\mathcal{M}_g$. In this case, one has that
$$s(D)=\frac{a}{\mbox{min}_{j=0}^{[g/2]}b_j}.$$ We denote by $s(\mm_g)$ the
\emph{slope of the moduli space $\mm_g$}, defined as
$$s(\mm_g):=\mbox{inf
}\{s(D):D\in \mbox{Eff}(\mm_g)\}.$$

\begin{proposition}\label{slope1}
We fix a moduli space $\mm_g$ with $g\geq 4$. If $s(\mm_g)<13/2$ then $\mm_g$ is of general type. If $s(\mm_g)>13/2$ then
the Kodaira dimension of $\mm_g$ is negative.
\end{proposition}
\begin{proof} If there exists $D\in \mbox{Eff}(\mm_g)$ with $s(D)<s(K_{\mm_g})$, it follows that
one can write $K_{\mm_g}\equiv \alpha\cdot \lambda +\beta \cdot D +\sum_{j=1}^{[g/2]} c_j\delta_j$, where
$\alpha, \beta>0$ and $c_j\geq 0$ for $1\leq j\leq [g/2]$. Since the class $\lambda \in \mbox{Eff}(\mm_g)$ is big,
we obtain that $K_{\mm_g}\in \mbox{int } \mbox{Eff}(\mm_g)$, hence by definition $\mm_g$ is a variety of general type.
\end{proof}
Any explicit calculation of a divisor class on $\mm_g$ provides an upper bound for $s(\mm_g)$. Estimating how small slopes of effective divisors on $\mm_g$ can be, is the subject of the Harris-Morrison Slope "Conjecture" \cite{HMo}:

\begin{conjecture}\label{slopeconjecture}
$$s(\mm_g)\geq
6+\frac{12}{g+1}.$$
\end{conjecture}

The conjecture would obviously imply that $\kappa(\mm_g)=-\infty$ for $g\leq 22$. However Conjecture 
\ref{slopeconjecture} is false and counterexamples have been found in \cite{FP}, \cite{F2}, \cite{F3}, \cite{Kh}-see also Section 6.1
of this paper.

There is a somewhat surprising connection between the Slope Conjecture and curves
sitting on $K3$ surfaces. This has been first observed in \cite{FP}:  Given
$g\geq 1$ we consider a Lefschetz pencil of curves of genus $g$
lying on a general $K3$ surface of degree $2g-2$ in $\PP^g$. This
gives rise to a curve $B$ in the moduli space $\mm_g$.  These pencils $B$ fill
up the entire moduli space $\mm_g$ for $g\leq 9$ or $g=11$ (see
\cite{Mu1}), and the divisor $\KK_{10}$  of curves lying on a $K3$ surface
for $g=10$. When $g\geq 13$, the pencils $B$ fill up the locus
$\KK_g\subset \mm_g$ of $K3$ sections of genus $g$ and $\mbox{dim}(\KK_g)=19+g$.

\begin{lemma}\label{nos1}
We have the formulas $B\cdot \lambda =g+1$, $B\cdot \delta_0=
6g+18$ and $B\cdot \delta_j=0$ for $j\neq 0$.
\end{lemma}

 It will turn out that the locus $\mathcal{K}_g$ becomes an obstruction for an effective divisor on $\mm_g$ to have small slope. The next result shows that in order to construct geometric divisors on $\mm_g$ having small slope, one must search for geometric
conditions which have codimension $1$ in moduli, and which are a relaxation of the condition that a curve be a section of a $K3$ surface. This philosophy governs the construction of Koszul divisors on $\mm_g$ carried out in \cite{F2}, \cite{F3}.
\begin{proposition}\label{nef}
Let $D$ be the closure in $\mm_g$ of an effective divisor on $\cM_g$. If the inequality $s(D)<6+12/(g+1)$ holds, then $D$ contains the locus $\KK_g$ of curves lying on $K3$ surfaces.
\end{proposition}
\begin{proof}
We consider as above the curve $B\subset
\mm_g$
corresponding to a Lefschetz
pencil of curves of genus $g$ on a general $K3$ surface $S$.
From
 Lemma \ref{nos1} we obtain that
$$\frac{B\cdot \delta}{B\cdot \lambda}=6+\frac{12}{g+1}>s(D),$$ which implies that $B \cdot D<0$ hence $B\subset D$. By varying both $B$ and $S$ inside the moduli space of polarized $K3$ surfaces, we obtain
 the inclusion $\KK_g\subset D$.
\end{proof}

Bounding  $s(\mm_g)$ from below, remains  one of the main open problems in the field. There is a straightforward (probably far from optimal) way of obtaining a bound on $s(\mm_g)$ by writing down any \emph{moving curve} $R\subset \mm_g$, that is, a curve which moves in an algebraic family $\{R_t\}_{t\in T}$
of curves on $\mm_g$ such that the set $\bigcup_{t\in T} R_t$ is dense in $\mm_g$. One instance of a moving curve is a \emph{complete intersection curve} $R=H_1\cap \cdots \cap H_{3g-4}$, where $H_i$ are numerically effective divisors on $\mm_g$.

If $R\subset \mm_g$ is a moving curve, then
 $R\cdot D\geq 0$, for any $D\in \mbox{Eff}(\mm_g)$, hence
$$s(\mm_g)\geq \frac{R\cdot \delta}{R\cdot \lambda}.$$ Obviously writing down and then computing the invariants of a moving curve in $\mm_g$ can be difficult.
An experimental bound $s(\mm_g)\geq O(1/g)$ was initially obtained in \cite{HMo} using Hurwitz schemes of covers of $\PP^1$. A similar (but nevertheless different) bound is obtained by D. Chen \cite{C} using covers of elliptic curves.
\subsection{Pandharipande's lower bound on $s(\mm_g)$.}\hfill

Recently, Pandharipande \cite{P} has found a short way of proving the inequality 
$$s(\mm_g)\geq O\bigl(\frac{1}{g}\bigr),$$ in a way that uses  only descendent integrals over
$\mm_{g, n}$  as well as some calculations on Hodge integrals that appeared in \cite{FaP}. We explain the main idea of his proof.

One constructs a covering curve for $\mm_{g}$ by pushing forward products of nef tautological classes on moduli spaces $\mm_{g, n}$ via the morphisms forgetting the marked points. In the simplest incarnation of this method, one considers the universal curve $\pi:\mm_{g, 1}\rightarrow \mm_g$ curve and the nef tautological class $\psi_1=c_1(\omega_{\pi})\in A^1(\mm_g)$. Then $\pi_*(\psi_1^{3g-4})\in A_1(\mm_g)$ is a covering curve class,
in particular for every divisor $$D\equiv a\lambda-\sum_{i=0}^{[g/2]} b_i\delta_i\in \mbox{Pic}(\mm_g)$$ which does not contain boundary components, we have that $\pi_*(\psi_1^{3g-4})\cdot D\geq 0$, hence,
$$s(D)\geq \frac{a}{b_0}\geq \frac{\pi_*(\psi_1^{3g-3})\cdot \delta_0}{\pi_*(\psi_1^{3g-3})\cdot \lambda}=\frac{\int_{\mm_{g, 1}}\delta_0\cdot \psi_1^{3g-3}}{\int_{\mm_{g, 1}} \lambda\cdot \psi_1^{3g-3}}.$$
We outline the calculation of the numerator appearing in this fraction.
 For the degree $2$ natural map $$\epsilon:\mm_{g-1, 3}\rightarrow \mm_{g, 1},\  \mbox{ }\ \epsilon([C, p, x, y]):=\bigl[\frac{C}{x\sim y}, p\bigr],$$ one has that $\epsilon_*([\mm_{g-1, 3}]=2\delta_0\in A^1(\mm_{g, 1})$, hence via the push-pull formula we find,
$$\int_{\mm_{g, 1}}\delta_0 \cdot \psi_1^{3g-3}=\frac{1}{2} \int_{\mm_{g-1, 3}}\psi_1^{3g-3}=\frac{1}{2}\int_{\mm_{g-1, 1}}\psi_1^{3g-5}.$$
The last equality here is an easy consequence of the \emph{string equation} \cite{W}
$$\int_{\mm_{g, n+1}} \psi_1^{a_1} \cdots \psi_n^{a_n}=\sum_{i=1}^n \int_{\mm_{g, n}} \psi_1^{a_1}\cdots \psi_i^{a_i-1}\cdots \psi_n^{a_n},$$
where $a_1, \ldots, a_n\geq 0$ such that $\sum_{i=1}^n a_i=3g-2+n$. 

The following evaluation follows by putting together \cite{M6} Section 6 and \cite{FaP} Section 3. For the sake of completeness we outline a proof:
\begin{lemma}\label{psiint}
$$\int_{\mm_{g, 1}} \psi_1^{3g-2}=\frac{1}{24^g\cdot  g!}.$$
\end{lemma}
\begin{proof}
The cokernel of the sheaf morphism $\pi^*(\mathbb E)\rightarrow \omega_{\pi}$ on $\mm_{g, 1}$ given by multiplication of global sections, is supported on the locus $X_2\cup \ldots \cup X_g$, where $X_j\subset \mm_{g, 1}$ is the closure of the subvariety  of pointed curves $[R\cup C_1\cup\ldots \cup C_j, p]$, where $R$ is a smooth rational curve,  $p\in R$ and $C_i$  are smooth
curves with $\#\bigl(R\cap C_i\bigr)=1$, for $1\leq i\leq j$ and $\sum_{i=1}^j g(C_i)=g$. Clearly $\mbox{dim}(X_j)=3g-2-j$, and there is a natural map $$f_j: X_j\rightarrow \mm_{0, j+1}$$ forgetting the tails $C_1, \ldots, C_j$, while retaining the intersection points $R\cap C_i$ for $1\leq i\leq j$. One has that $\psi_{1 |\ X_j}=f_j^*(\psi_p)$, where $\psi_p\in A^1(\mm_{0, j+1})$ denotes the
cotangent line class on $\mm_{0, j+1}$ corresponding to the marked point labeled by $p\in R$. For dimension reasons it follows that $\psi^{g-1}_{1 | \ X_2\cup \ldots \cup X_g}=0$, whereas
$\psi^{g-2}_{1 | \ X_2\cup \ldots \cup X_g}$ must be entirely supported on the locus $X_g$. Putting these observations together, one finds that
\begin{equation}\label{inter}
\Bigl(\frac{\pi^*c(\mathbb E)}{1+\psi_1}\Bigr)_{2g-2}=\psi_1^{g-2}[X_g]_{\mathbb Q}.
\end{equation}
To estimate $\int_{\mm_{g, 1}} \psi_1^{3g-2}$, one uses  Mumford's relation $c(\mathbb E)\cdot c(\mathbb E^{\vee})=1$, cf. \cite{M6}. This comes from the exact sequence which globalizes Serre duality
 $$0\longrightarrow \mathbb E\rightarrow R^1\pi_*\Omega_{\pi}^1\rightarrow \mathbb E^{\vee}\rightarrow 0,$$
 where the rank $2g$ vector bundle in the middle possesses a Gauss-Manin connection.
Accordingly, we can write that
$$\int_{\mm_{g, 1}}\psi_1^{3g-2}=\Bigl(\frac{\pi^*(\mathbb E)}{1+\psi_1}\Bigr)_{2g-2}\cdot (\lambda_g+\lambda_{g-1}\psi_1+\cdots+\psi_1^g)
=\lambda_{g}\psi_1^{g-2}\cdot [X_{g}]_{\mathbb Q}.$$
This last intersection number can be evaluated via the map of degree $g!$,
$$\phi:(\mm_{1, 1})^g \times \mm_{0, g+1}\rightarrow X_{g},$$
which attaches $g$ elliptic tails at the first $g$ marked points of a rational $(g+1)$-pointed stable curve. Clearly $\phi^*(\mathbb E_{| X_g})=\mathbb E_1\boxtimes \cdots \boxtimes \mathbb E_1$, where $\mathbb E_1$ is the Hodge bundle on $\mm_{1, 1}$. Since $\int_{\mm_{1, 1}}\lambda_1 =1/24$,  one finds that,
$$\lambda_{g}\psi_1^{g-2}\cdot [X_g]_{\mathbb Q}=\frac{1}{g!}\Bigl(\int_{\mm_{1, 1}} \lambda_1\Bigr)^{g}\cdot \int_{\mm_{0, g+1}}\psi_1^{g-2}=\frac{1}{ 24^g\cdot g!}.$$
\end{proof}
To evaluate the integral $\int_{\mm_{g, 1}} \lambda \cdot \psi_1^{3g-3}$, first one uses the GRR calculation of $\mbox{ch}(\mathbb E)$ applied to the universal curve $\pi:\mm_{g, 1}\rightarrow \mm_g$.  One finds that
$$\int_{\mm_{g, 1}}\lambda\cdot \psi_1^{3g-3}=\frac{1}{12}\int_{\mm_{g, 2}} \psi_1^{3g-3}\psi_2^2-\frac{1}{12} \int_{\mm_{g, 1}}\psi_1^{3g-2}+\frac{1}{24} \int_{\mm_{g-1, 3}} \psi_{1}^{3g-3}.$$ The last integral is evaluated using again the string equation, for the middle one we use Lemma
\ref{psiint}. The first integral is evaluated using \cite{W} and one finally proves:

\begin{theorem}\label{pand}
$$s(\mm_g)\geq \frac{\int_{\mm_{g, 1}}\delta_0\cdot \psi_1^{3g-3}}{\int_{\mm_{g, 1}} \lambda\cdot \psi_1^{3g-3}}=\frac{60}{g+4}.$$
\end{theorem}
Note that the bound $O(1/g)$ obtained in this theorem is quite similar to the experimental bound $\frac{576}{5g}$ obtained in \cite{HMo} using Hurwitz covers.

\begin{remark} Another very natural covering curve for $\mm_g$, which potentially could produce a much better lower bound for $s(\mm_g)$ than the one in \cite{P}, has been recently proposed by Coskun, Harris and Starr \cite{CHS}: If $\textbf{Hilb}_{g, 1}$ denotes the Hilbert scheme of canonically embedded curves $C\subset \PP^{g-1}$, then $\mbox{dim } \textbf{Hilb}_{g, 1}=g^2+3g-4$. We denote by $r(g)$ the largest number $r$, such that through $r$ general points in $\PP^{g-1}$ there passes a canonical curve $[C\hookrightarrow \PP^{g-1}]\in \textbf{Hilb}_{g, 1}$. It has been determined in \cite{St} that $$r(g)=g+5+\bigl[\frac{6}{g-2}\bigr]$$ (this is, the smallest solution of the necessary inequality $\mbox{dim } \textbf{Hilb}_{g, 1}+r\geq (g-1)r$). In particular $r(3)=14$ (as it should be!) and $r(5)=12$. For $g\geq 9$, one fixes general points
$p_1, \ldots, p_{g+5} \in \PP^{g-1}$ as well as a general linear space $\PP^{g-7}\subset \PP^{g-1}$. The family $X_g\subset \mm_g$ consisting of canonical curves
$[C\hookrightarrow \PP^{g-1}]\in \textbf{Hilb}_{g, 1}$ passing through $p_1, \ldots, p_{g+5}$, and such that $C\cap \PP^{g-7}\neq \emptyset$ is a covering curve for $\mm_g$. It is an interesting problem to determine the slope $X_g\cdot\delta/X_g\cdot \lambda$.
\end{remark}

\section{The Harris-Mumford theorem revisited: An alternative proof via syzygies}

In this section we present a different proof of the main result from \cite{HM} by replacing the calculation of the
class of the Hurwitz divisor $\mm_{g, k}^1$ of $k$-gonal curves of genus $g=2k-1$ by the calculation of the class of a certain  \emph{Koszul divisor}
$\overline{\cZ}_{g, k-2}$, consisting of canonical curves $[C]\in \cM_g$ with extra syzygies at the $(k-2)$-nd step in its minimal
 graded resolution. The advantage of this approach is that the proof that
$\mm_g$ is of general type becomes shorter since one can completely avoid having to develop the theory of admissible covers and do without the enumerative calculations that occupy
a large part of \cite{HM}, precisely pg. 53-86, or alternatively, develop the theory of limit linear series \cite{EH1}. The proof becomes also more direct and logical, since it uses solely the geometry of canonical curves of genus $g$ and that of the corresponding Hodge bundles on $\mm_g$ , rather than
the geometry of an auxiliary Hurwitz stack. The disadvantage of this approach, is that the statement that the locus $\overline{\cZ}_{g, k-2}$ is a divisor on $\mm_g$ is highly non-trivial and it is equivalent to Green's Conjecture for generic curves of odd genus (Voisin's theorem \cite{V1}, \cite{V2}). This situation is somewhat similar to that encountered in the streamlined proof of Theorem \ref{hme} presented by Eisenbud and Harris in \cite{EH3} (and which is comparable in length to our proof): Showing that the a priori virtual Brill-Noether locus is an actual divisor in $\mm_g$, requires the full force of the Brill-Noether theory and is arguably more difficult than computing the class of the Brill-Noether divisor on $\mm_g$.

We start by recalling a few basic facts on syzygies. For a smooth curve $C$ and a globally generated line bundle $L\in \mbox{Pic}^d(C)$, we denote
by $K_{i, j}(C, L)$ the Koszul cohomology group obtained from the
complex
$$\wedge^{i+1} H^0(L)\otimes H^0(L^{\otimes
(j-1)})\stackrel{d_{i+1, j-1}}\longrightarrow \wedge^i H^0(L)\otimes H^0(L^{\otimes
j})\stackrel{d_{i, j}}\longrightarrow \wedge^{i-1} H^0(L)\otimes H^0(L^{\otimes
(j+1)}),$$ where the maps $d_{i, j}$ are the Koszul differentials defined by
(cf.
\cite{La2}, \cite{PR})
$$d_{i, j}\bigl(f_1\wedge \ldots \wedge f_i\ \otimes u\bigr):=\sum_{l=0}^i (-1)^l \bigl(f_1\wedge \ldots \wedge \check{f_l}\ldots \wedge f_i\bigr)\otimes (uf_l),$$
with $f_l\in H^0(C, L)$ and $u\in H^0(C, L^{\otimes j})$. If $R(C, L):=\oplus_{n\geq 0} H^0(C, L^{\otimes n})$ denotes the
graded module over the polynomial ring $S:=\mathrm{Sym }\ H^0(C, L)$, then
$$K_{i, j}(C, L)=\mathrm{Tor}_i^S\bigl(\mathbb C, R(C, L)\bigr)_{i+j}.$$ There is a standard geometric way of computing Koszul cohomology groups using \emph{Lazarsfeld bundles}. Since $L$ is assumed to be globally generated, we can define the vector bundle $M_L$ on $C$ through the following exact sequence on $C$:
$$ 0\rightarrow M_L\rightarrow H^0(L)\otimes \OO_C \rightarrow L\rightarrow
0.$$  A diagram chasing argument using the exact sequences
$$ 0\longrightarrow \wedge^a M_L\otimes L^{\otimes b} \rightarrow
\wedge^a H^0(L)\otimes L^{\otimes b}\longrightarrow \wedge^{a-1}
M_L\otimes L^{\otimes (b+1)}\longrightarrow 0$$ for various $a$ and
$b$, shows that there is an identification cf. \cite{La2}
\begin{equation}\label{koszul}
 K_{a, b}(C, L)=\frac{H^0(C, \wedge^a
M_L\otimes L^{\otimes b})}{\mbox{Image} \{\wedge^{a+1} H^0(C, L)\otimes
H^0(C, L^{\otimes (b-1)})\}}\ .
\end{equation}
\begin{example}
From (\ref{koszul}) we find that $K_{0, 2}(C, L)=0$ if and only if the multiplication map $\mbox{Sym}^2 H^0(C, L)\rightarrow H^0(C, L^{\otimes 2})$ is surjective. Assuming $L$ is normally generated, we have that $K_{1, 2}(C, L)=0$ if and only if $C\stackrel{|L|}\rightarrow \PP\bigl(H^0(C, L)^{\vee}\bigr)$ is cut out by quadrics. More generally, one says that $L$ satisfies the \emph{Green-Lazarsfeld property $(N_p)$} for $p\geq 0$, if the vanishing $K_{i, 2}(C, L)=0$ holds for all $0\leq i\leq p$. This corresponds intuitively to the situation that the first $p$ syzygies of the image curve $C\stackrel{|L|}\rightarrow \PP\bigl(H^0(C, L)^{\vee}\bigr)$ are as simple as possible, that is, linear.
\end{example}

From now on we specialize to the case $L=K_C\in W_{2g-2}^{g-1}(C)$ and we consider the canonical map
$C\stackrel{|K_C|}\longrightarrow \PP^{g-1}$. If $C$ is non-hyperreliptic, we set $\I_{C/\PP^{g-1}}$ to be the ideal of the canonically embedded curve.

\begin{proposition}\label{determ}
For any non-hyperrelliptic curve $[C]\in \cM_g$ and any integer $0\leq i\leq (g-1)/2$ we have the following equivalence: $$K_{i, 2}(C, K_C)\neq 0\Longleftrightarrow
h^0\bigl(\PP^{g-1}, \Omega^i_{\PP^{g-1}} (i+2) \otimes \I_{C/\PP^{g-1}}\bigr)\geq
{g-1\choose i+2}\frac{(g-2i-3)(i+1)}{g-i-1}+1.$$
\end{proposition}
\begin{proof}
We start with a canonically embedded curve $C\stackrel{|K_C|}\hookrightarrow \PP^{g-1}$. Throughout the proof we use the identification $M_{\PP^{g-1}}=\Omega_{\PP^{g-1}}(1)$ coming from the Euler sequence on $\PP^{g-1}$. Since the vector bundle $M_{K_C}$ is stable (cf. \cite{PR} Corollary 3.5), we have the vanishing $$H^1(C, \wedge^{i}\otimes K_C^{\otimes 2})=0$$ because $\mu\bigl(\wedge^i M_{K_C}\otimes K_C^{\otimes 2})>2g-1$. It follows from (\ref{koszul}) that  $K_{i, 2}(C, K_C)\neq 0$ if and only if the map
$$H^1(C, \wedge^{i+1}M_{K_C}\otimes K_C)\rightarrow \wedge^{i+1}H^0(C, K_C)\otimes
H^1(C, K_C)$$ is an isomorphism, or equivalently $h^1(C, \wedge^{i+1}
M_{K_C}\otimes K_C)={g\choose
i+1}$. We write down the following commutative diagram, where by abusing notation, we shall denote by the same letter a sheaf morphism and the group morphism it induces at the level of global sections:
$$\begin{array}{ccccc}
0 & \; & 0 & \;&  0\\
\rmapdown{} & \; & \rmapdown{} &\; & \rmapdown{}\\
 \wedge^{i+1}M_{\PP^{g-1}}\otimes\I_{C/\PP^{g-1}}(1) & \longrightarrow & \wedge^{i+1}H^0(\OO_{\PP^{g-1}}(1))\otimes \I_{C/\PP^{g-1}}(1) & \longrightarrow & \wedge^i M_{\PP^{g-1}}\otimes \I_{C/\PP^{g-1}}(2)  \\
 \rmapdown{} & \; & \rmapdown{} & \; & \rmapdown{}  \\
  \wedge^{i+1} M_{\PP^{g-1}}(1) & \longrightarrow  & \wedge^{i+1} H^0(\OO_{\PP^{g-1}}(1))\otimes \OO_{\PP^{g-1}}(1) & \longrightarrow & \wedge^i M_{\PP^{g-1}}(2) \\
 \rmapdown{\alpha} & \; &\rmapdown{\beta} & \; &\rmapdown{\gamma} \\
 \wedge^{i+1}M_{K_C}\otimes K_C &  \longrightarrow & \wedge^{i+1} H^0(K_C)\otimes K_C & \longrightarrow &
\wedge^i M_{K_C}\otimes K_C^{\otimes 2}  \\
\rmapdown{} &\; & \rmapdown{} & \; & \rmapdown{}\\
0 & \;&  0 & \;&  0\\
\end{array}$$
Applying the Snake Lemma, we find that $H^0(\PP^{g-1}, \wedge^i M_{\PP^{g-1}}\otimes \I_{C/\PP^{g-1}}(2))=\mathrm{Coker}(\alpha)$. We also note that $h^0\bigl(\PP^{g-1}, \wedge^{i+1} M_{\PP^{g-1}}(1)\bigr)={g\choose i+2}$ (use for instance Bott's vanishing theorem). Thus the condition $K_{i, 2}(C, K_C)=0$ is satisfied if and only if
$$\mathrm{dim }\ \mathrm{ Coker}(\alpha)=h^0(C, \wedge^{i+1} M_{K_C}\otimes K_C)-h^0\bigl(\PP^{g-1}, \wedge^{i+1} M_{\PP^{g-1}}(1)\bigr)=$$
$$={g-1\choose i+1}(g-2i-3)+h^1(C, \wedge^{i+1} M_{K_C}\otimes K_C)-{g\choose i+2}\leq $$
$$\leq {g-1\choose i+2}\frac{(g-2i-3)(i+1)}{g-i-1}.$$
\end{proof}

For $g=2i+3$, we find that $K_{i, 2}(C, K_C)\neq 0$ if and only if the map
\begin{equation}\label{koszet}
H^0(\PP^{g-1}, \wedge^i M_{\PP^{g-1}} (2))\stackrel{\gamma}\longrightarrow H^0(C, \wedge^i M_{K_C}\otimes K_C^{\otimes 2})
\end{equation}
is not an isomorphism. We note that $\gamma$ is a map between vector spaces of the same dimensions:
$$h^0\bigl(\PP^{g-2}, \wedge^i M_{\PP^{g-1}}(2)\bigr)=(i+1){g+1\choose i+2}=\chi(C, \wedge^i M_{K_C}\otimes K_C^{\otimes 2})=h^0(C, \wedge^i M_{K_C}\otimes K_C^{\otimes 2})$$ (for the left hand side use  Bott vanishing, for the right hand-side the Riemann-Roch theorem.) This shows that the locus
$$\cZ_{g, i}:=\{[C]\in \cM_g: K_{i, 2}(C, K_C)\neq 0\},$$
being the degeneracy locus of a morphism between two vector bundles of the same rank over $\cM_g$, is a virtual divisor on the moduli space of curves.

\begin{example} By specializing to the case  $g=3$, we find the following interpretation
 $$\cZ_{3, 0}:=\{[C]\in \cM_3: \mbox{Sym}^2 H^0(C, K_C)\rightarrow H^0(C, K_C^{\otimes 2}) \mbox{ is not an isomrphism}\}.$$
Using M. Noether's theorem \cite{ACGH}, it follows that $\cZ_{3, 0}$ consists precisely of hyperelliptic curves, that is, $\mbox{supp}(\cZ_{3, 0})=\mbox{supp}(\cM_{3, 2}^1)$. In the next case $g=5$, we use Petri's theorem stating
that a non-hyperelliptic canonical curve $[C]\in \cM_5$ is cut out by quadrics unless it has a $\mathfrak g^1_3$.
We obtain that $\mbox{supp}(\cZ_{5, 1})=\mbox{supp}(\cM_{5, 3}^1)$.
\end{example}

In order to describe the closure $\overline{\cZ}_{g, i}$ of $\cZ_{g, i}$ inside $\mm_g$, we shall extend the determinantal description of $\cZ_{g, i}$ over a partial compactification
of $\textbf{M}_g$. Our methods seem well-suited for divisor class calculations but harder to implement in the case
of Koszul cycles on $\mm_g$ of higher codimension.

We denote by
$\textbf{M}_g^*:=\textbf{M}_g\cup \bigl(\cup_{j=0}^{i+1}
\Delta_j^0\bigr)$ the locally closed substack of $\rem_g$ defined as
the union $\textbf{M}_g$ and the open
substacks $\Delta_j^0\subset \Delta_j$ for $1\leq j\leq i+1$
consisting of $1$-nodal genus $g$ curves $[C\cup_y D]$, with $[C, y]\in
\cM_{g-j, 1}$ and $[D, y]\in \cM_{j, 1}$, that is, $\Delta_j^0$ is the intersection of $\Delta_j$ with the codimension $1$ stratum in the topological stratification of $\overline{\textbf{M}}_g$.  The substack $\Delta_0^0\subset \Delta_0$ classifies
$1$-nodal irreducible genus $g$ curves $$\bigl[C_{yq}:=\frac{C}{q\sim y}\bigr]\in \mm_g,$$ where $[C,
q, y]\in \cM_{g-1, 2}$ together with their degenerations consisting of unions
of a smooth genus $g-1$ curve and a nodal rational curve. We set $\pem_g:=\textbf{M}_g\cup \Delta_0^0\cup
\Delta_1^0\subset \textbf{M}_g^*$.

For integers $0\leq a\leq i$ and $ b\geq 2$ we define
vector bundles $\G_{a,b}$ over $\pem_g$ with fibre
$$\G_{a,b}[C]=H^0(C, \wedge^a M_{K_C}\otimes K_C^{\otimes b})$$ over every point $[C]\in \cM_g$. The question is of course how the extend this description of $\G_{a, b}$ over the locus of stable curves. In this paper
we shall only describe how to construct the bundles $\G_{a, b}$ over $\pem_g$, which will suffice in order to compute the slope of $\overline{\cZ}_{g, i}$ and prove Theorem \ref{hme} for $g=2i+3$. For full details on how to extend the vector bundles $\G_{a, b}$ over $\pem_g$ (that is, outside codimension $2$ over all the boundary divisors on $\mm_g$), we refer to \cite{F2} p. 75-86. We start by constructing the bundles $\G_{0, b}$:
\begin{proposition}\label{fibre2} For each $b\geq 2$ there
exists a vector bundle $\G_{0,b}$ over $\pem_g$ of rank $(2b-1)(g-1)$ with fibres admitting the
following description:
\begin{itemize}
\item For $[C]\in \cM_g$ we have that
$\G_{0,b}[C]=H^0(C, K_C^{\otimes b})$.
\item For $[C\cup_y
E])\in \Delta_1^0$, where $E$ is an elliptic curve, if $u\in H^0(C, K_C\otimes \OO_C(2y))-H^0(C, K_C)$ denotes any meromorphic $1$-form with non-zero residue at $y$, then
$$\G_{0,b}[C\cup_y E]=H^0(C, K_C^{\otimes b}\otimes \OO_C((2b-2)\cdot y))+\mathbb C\cdot
u^b\subset H^0\bigr(C, K_C^{\otimes b}\otimes \OO_C(2b\cdot y)\bigl).$$
\item For $[C_{yq}=C/y\sim
q]\in \Delta_0^0$, where $q,y\in C$ and $u\in H^0(C, K_C\otimes \OO_C(y+q))-H^0(C, K_C)$ is a
meromorphic $1$-form with non-zero residues at $y$ and $q$, we have that
$$\G_{0,b}[C_{yq}]=H^0\bigl(C, K_C^{\otimes b}\otimes \OO_C((b-1)y+(b-1)q)\bigr)\oplus \mathbb C\cdot u^b\subset H^0\bigl(C, K_C^{\otimes b}\otimes \OO_C(by+bq)\bigr).$$
\end{itemize}
\end{proposition}

The idea to define the vector bundles $\G_{0, b}$ as suitable twists by boundary divisors of powers of the relative dualizing sheaf over the universal curve, that is, $$\G_{0, b}=\pi_*\bigl(\omega_{\pi}^{\otimes b}\otimes \sum_{j=1}^{[g/2]}  \pi^*(\OO_{\rem_g}(c_j^b \ \delta_j))\bigr)$$ for precisely determined constants  $c_j^b\geq 0$, comes of course from the theory of limit linear series. Recalling that $\sigma:
\mathfrak G^{g-1}_{2g-2}\rightarrow \pem_g$ denotes the stack of limit $\mathfrak g^{g-1}_{2g-2}$'s,  then for a
curve $[C\cup_y E]\in \Delta_1^0$, the fibre $\sigma^{-1}[C\cup_y E]$ consists of a single limit linear series
$$\Bigl(l_C=\bigl(\omega_C(2y), H^0(\omega_C(2y)\bigr)\in G^{g-1}_{2g-2}(C), \ l_E=\bigl(\OO_E(2g y), (g-1) y+|(g-1)y|\bigr)\in G^{g-1}_{2g-2}(E)\bigr)\Bigr).$$ The bundle $\G_{0, 1}$ retains the aspect of this limit $\mathfrak g^{g-1}_{2g-2}$ corresponding to the component of genus $g-1$, while dropping the information coming from the elliptic tail. Similarly, for $b\geq 2$, it is an easy exercise in limit linear series to show the fibre $\G_{0, b}[C\cup_y E]$ is precisely the $C$-aspect of the limit $\mathfrak g_{(2b-1)(g-1)-1}^{2b(g-1)}$ induced from $\omega_{C\cup E}^{\otimes b}$. The situation becomes more complicated when extending $\G_{0, b}$ over the entire
stack $\textbf{M}_g^*$. As explained in \cite{F2} Theorem 3.13 in the case of the Hurwitz stack (and the same holds true for $\textbf{M}_g^*$ itself), the twisting coefficients $c_j^b$ are chosen in a unique way such that the resulting bundles $\G_{a, b}$ fit in exact sequences of type (\ref{gi}).

Having defined  $\G_{0, b}$ we now define
inductively all vector bundles $\G_{a, b}$. First we define
$\G_{1, b}$ as the kernel of the multiplication map $\G_{0,
1}\otimes \G_{0, b}\rightarrow \G_{0, b+1}$, that is, by the exact
sequence
$$0\longrightarrow \G_{1, b}\longrightarrow \G_{0,1}\otimes \G_{0, b}\longrightarrow \G_{0, b+1}\longrightarrow 0.$$
Having defined $\G_{l, b}$ for all $l\leq a-1$, the
vector bundle $\G_{a, b}$ is defined through the following exact sequence over $\pem_g$:
\begin{equation}\label{gi}
0\longrightarrow \G_{a, b}\longrightarrow \wedge^a \G_{0,
1}\otimes \G_{0, b}\stackrel{\phi_{a, b}}\longrightarrow \G_{a-1,
b+1}\longrightarrow 0.
\end{equation}
\begin{proposition}
The Koszul maps $\phi_{a, b}:\wedge^a \G_{0, 1}\otimes \G_{0, b}\rightarrow \G_{a-1, b+1}$ are well-defined and surjective for all integers $b\geq 2$ and $0\leq a\leq b$. In particular the exact sequences (\ref{gi})
make sense and the vector bundles $\G_{a, b}$ are well-defined.
\end{proposition}
\begin{proof} This proof is similar to \cite{F2} Proposition 3.10. We use that the vector bundle $M_{K_C\otimes \OO_C(y+q)}$ is semi-stable for $[C, y, q]\in \cM_{g-1, 2}$, in particular $$H^1\bigl(C, \wedge^a M_{K_C\otimes \OO_C(y+q)}\otimes K_C^{\otimes b}((b-1)\cdot (y+q))\bigr)=0,$$ that is the map
$$\wedge^a H^0\bigl(K_C(y+q)\bigr)\otimes H^0\bigl(K_C^{\otimes b}((b-1)(y+q))\bigr)\rightarrow H^0\bigl(\wedge^{a-1} M_{K_C\otimes \OO_C(y+q)}\otimes K_C^{\otimes (b+1)}(b(y+q))\bigr)$$
is surjective. The rest now follows from the description of the fibres of the bundles $\G_{a, b}$ provided in Proposition \ref{fibre2}.
\end{proof}
For $0\leq a\leq i$ and $b\geq 1$ we define vector bundles
$\H_{a, b}$ over $\pem_g$ having
fibre $$\H_{a, b}[C]=H^0\bigl(\PP^{g-1}, \wedge^a M_{\PP^{g-1}}\otimes \OO_{\PP^{g-1}}(b)\bigr)$$ over each
point corresponding to a smooth curve $[C]\in \cM_g$ with the canonical map $C\stackrel{|K_C|}\longrightarrow \PP^{g-1}$.  First we set
$\H_{0, b}:=\mbox{Sym}^b( \mathbb E\otimes \OO_{\pem_g}(\delta_1))$ for $b\geq 1$.
Having already defined $\H_{a-1, b}$ for all $b\geq 1$,  we define $\H_{a, b}$
via the exact sequence
\begin{equation}\label{sym}
0\longrightarrow \H_{a, b}\longrightarrow \wedge^a \H_{0,
1}\otimes \mbox{Sym}^b \H_{0, 1}\longrightarrow \H_{a-1,
b+1}\longrightarrow 0.
\end{equation}
Note that the bundles $\H_{a, b}$ are defined entirely in terms of
the Hodge bundle $\mathbb E$. There is a natural vector bundle morphism over $\pem_g$ $$\gamma_{a,
b}:\H_{a, b}\rightarrow \G_{a, b}.$$  When $g=2i+3$ then
$\mbox{rank}(\H_{i, 2})=\mbox{rank}(\G_{i, 2})$ and the degeneracy
locus $Z(\gamma_{i, 2})$ of the morphism $\gamma_{i, 2}$ is a codimension $1$ compactification in $\pem_g$ of the locus $\cZ_{g, i}$.

We shall determine the class $c_1(\G_{i, 2}-\H_{i, 2})\in \mathrm{Pic}(\pem_g)$
by computing its intersection it with the following test
curves lying in the boundary of $\mm_g$: We fix a pointed curve $[C, q]\in \cM_{g-1, 1}$ and a general elliptic curve
$[E, y]\in \cM_{1, 1}$. We define two $1$-parameter families
\begin{equation}\label{testcurves}
C^0:=\{\frac{C}{y\sim q}: y\in C\}\subset \Delta_0 \subset \mm_{g}
\mbox{ and }C^1:=\{C\cup _y E: y\in C\}\subset \Delta_1\subset
\mm_{g}.
\end{equation}
These families intersect the generators of
$\mbox{Pic}(\rem_g)$ as follows (cf. \cite{HM} pg. 83-85):
$$
C^0\cdot \lambda=0,\ C^0\cdot \delta_0=-2g+2, \ C^0\cdot
\delta_1=1 \mbox{ and } C^0\cdot \delta_a=0\mbox{ for }a\geq 2,
\mbox{ and}$$
$$C^1\cdot \lambda=0, \ C^1\cdot \delta_0=0, \ C^1\cdot
\delta_1=-2g+4, \ C^1\cdot \delta_a=0 \mbox{ for }a\geq 2.$$
\begin{lemma}\label{inter1}
We fix $[C, q]\in \cM_{g-1}$ and we consider the test curves $C^0, C^1\subset \mm_g$. Then for all integers $j\geq 1$
the following formulas:
\begin{enumerate}
\item $C^1\cdot c_1(\G_{0, j})=-2g+4$.
\item $C^0\cdot c_1(\G_{0, j})=(j-1)\bigl(j(g-1)+j-1\bigr)+j$.
\end{enumerate}
\end{lemma}
\begin{proof} We denote by $p_1, p_2:C\times C\rightarrow C$ the two projections and $\Delta\subset C\times C$
is the diagonal. We give details only for the first
calculation the remaining one being similar. We have the identification $\G_{0, 1 |C^1}=(\pi_1)_*\bigl(\pi_2^*
(K_C)\otimes \OO(2\Delta)\bigr)$, from which we obtain that $c_1(\G_{0, 1 |C^1})=-2g+4$. For $j\geq 2$ we use the following exact sequences of bundles on $C$:
$$0\longrightarrow (\pi_1)_*\Bigl(\pi_2^*(K_C^{\otimes j})\otimes \OO((2j-2)\Delta)\Bigr) \longrightarrow \G_{0, j| C^1}\longrightarrow (\pi_1)_*\Bigl(\pi_2^*(K_C^{\otimes j})\otimes \OO_{\Delta} (2j\Delta)\Bigr)\longrightarrow 0.$$
An immediate application of Grothendieck-Riemann-Roch for the projection morphism $p_1:C\times C\rightarrow C$ gives that $$c_1(\pi_1)_*\Bigl(\pi_2^*(K_C^{\otimes j})\otimes \OO_{C\times C}((2j-2)\Delta)\Bigr)=2(g-1)(j-1),$$ which finishes the proof.
\end{proof}
\begin{theorem}\label{clasa}
The class of the virtual divisor $\overline{\cZ}_{2i+3, i}$ in $\mathrm{Pic}(\pem_{2i+3})$ equals
$$[\overline{\cZ}_{2i+3, i}]^{virt}=c_1(\G_{i, 2}-\H_{i, 2})=\frac{1}{i+2}{2i\choose i}\Bigl(6(i+3)\lambda-(i+2)\delta_0-6(i+1)\delta_1\Bigr).$$
\end{theorem}

\begin{proof} We have
constructed the vector bundle morphism $\gamma_{i, 2}:\H_{i,
2}\rightarrow \G_{i, 2}$ over the stack $\pem_g$. For $g=2i+3$ we know that
$\mbox{rank}(\H_{i, 2})=\mbox{rank}(\G_{i, 2})$ and the virtual
Koszul class $[\overline{\mathcal{Z}}_{g, i}]^{virt}$ equals
$c_1(\G_{i, 2}-\H_{i, 2})$. We recall that for a rank $e$ vector
bundle $\mathcal{E}$ over a stack $X$ and for $i\geq 1$, we have
formulas $$c_1(\wedge^i \mathcal{E})={e-1\choose i-1}c_1(\mathcal{E})\ \mbox{ and } \ c_1(\mbox{Sym}^i \mathcal{E})={e+i-1\choose e}c_1(\mathcal{E}).$$ We write
$c_1(\G_{i, 2}-\H_{i, 2})=\mathfrak{a} \lambda-\mathfrak{b}_0\delta_0-\mathfrak{b}_1 \delta_1$. Using the exact sequences (\ref{gi})
we find that
$$c_1(\G_{i, 2})= \sum_{l=0}^i (-1)^l c_1(\wedge^{i-l}\G_{0,
1}\otimes \G_{0, l+2})=\sum_{l=0}^i (-1)^l{ g \choose i-l}
c_1(\G_{0, l+2})+$$ $$+\sum_{l=0}^i (-1)^l
\bigl((g-1)(2l+3)\bigr){g-1 \choose i-l-1}c_1(\G_{0, 1}).$$
Using repeatedly the exact sequence (\ref{sym}) we find that
$$c_1(\H_{i, 2})=\sum_{l=0}^i (-1)^l c_1(\wedge^{i-l} \H_{0,
1}\otimes \mbox{Sym}^{l+2} \H_{0, 1})=$$ $$=\sum_{l=0}^i (-1)^l{g\choose
i-l}c_1(\mbox{Sym}^{l+2}(\H_{0, 1}))+\sum_{l=0}^i (-1)^l{g+l+1\choose
l+2}c_1(\wedge^{i-l}\H_{0, 1})$$
$$=\sum_{l=0}^i (-1)^l\Bigl( {g \choose i-l}{g+l+1 \choose
g}+{g+l+1\choose l+2}{g-1 \choose i-l-1}\Bigr)c_1(\H_{0, 1})=$$
$$=4(2i+1){2i \choose i} c_1(\H_{0, 1}),$$ with
 $\G_{0, 1}=\H_{0, 1}=\mathbb E \otimes \OO_{\pem _g}(\delta_1)$. We intersect both these formulas with the test curves $C^0$ and $C^1$ and write that
 $$(2g-2)\mathfrak{b}_0-\mathfrak{b}_1=C^0\cdot [\overline{\cZ}_{g, i}]^{virt}=(i+1){2i+2\choose i}\ \mbox{ and} $$
 $$(2g-4)\mathfrak{b}_1=C^1\cdot [\overline{\cZ}_{g, i}]^{virt}=6(i+1){2i+2\choose i}.$$
 These relations determine $\mathfrak b_0$ and $\mathfrak b_1$.
 Finally we claim that we also have the relation $\mathfrak{a}-12\mathfrak{b}_0+\mathfrak{b}_1=0$ which finishes the proof.
 Indeed, we consider $q$ the curve $R\subset \mm_g$ obtained by attaching to a fixed point $q\in C$ of a curve of genus $g-1$ a Lefschetz pencil of plane cubics. Then
 $R\cdot \lambda=1, \ R\cdot \delta_0=12, \ R\cdot
\delta_1=-1$ and $R\cdot \delta_j=0$ for $j\geq 2$. Then $$\mathfrak{a}-12\mathfrak{b}_0+\mathfrak{b}_1=0=R\cdot c_1(\G_{i, 2}-\H_{i, 2})=0,$$ and this follows because  $\G_{a, b | R}$ are numerically trivial (It is clear that $\G_{0, b| R}$ are trivial for $b\geq 1$ and then one uses
(\ref{gi}) and (\ref{sym})).
\end{proof}
\begin{example}
For $i=0$ hence $g=3$, Theorem \ref{clasa} reads like $$[\overline{\cZ}_{3, 0}]^{virt}= c_1(\G_{0, 2}-\mathrm{Sym}^2 \G_{0, 1})=9\lambda-\delta_0-3\delta_1\equiv \mm_{3, 2}^1\in \mathrm{Pic}(\mm_3).$$ Thus our calculation yields a computation of the  compactified  divisor $\mm_{3, 2}^1$
on $\mm_3$ of hyperelliptic curves. Thus we have the relation $Z(\gamma_{0, 2})=\overline{\cZ}_{3, 0}$ and the vector bundle morphism
$\gamma_{0, 2}: \H_{0, 2}\rightarrow \G_{0, 2}$ provides the "correct" determinantal structure of the compactification of the hyperelliptic divisor. A different
compactification of $\cM_{3, 2}^1$ is provided by the vector bundle morphism between Hodge bundles $$\chi_3: \mathrm{Sym}^2(\mathbb E_1)\rightarrow
\mathbb E_2, \mbox{  } \ \mbox{ } \chi_3[X]: \mbox{Sym}^2 H^0(X, \omega_X)\rightarrow H^0(X, \omega_X^{\otimes 2})$$
 for $[X]\in \mm_3$. The class of its degeneration locus is  $c_1(\mathbb E_2-\mathrm{Sym}^2 \mathbb E_1)=9\lambda-\delta_0-\delta_1$ (use Remark \ref{higherhodge}). It follows that there is an equality of cycles
$$Z(\chi_{3})=Z(\gamma_{0, 2})+2\delta_1\in A^1(\rem_3),$$
that is, $\chi_{3}$ is an everywhere degenerate morphism along the divisor $\Delta_1$.
This discussion in low genus, already indicates that the determinantal structure induced by the morphism $\gamma_{i, 2}:\H_{i, 2}\rightarrow \G_{i, 2}$ provides the
right compactification of $\cZ_{g, i}$ over $\pem_g$.
\end{example}

In a way analogous to \cite{F2}, one can extend the vector bundles $\G_{a, b}$ and $\H_{a, b}$ as well as the vector bundle morphism $\gamma_{a, b}: \H_{a, b}\rightarrow \G_{a, b}$  over the
larger codimension $1$ compactification $\textbf{M}_g^*$, in a way that the exact sequence (\ref{gi}) and (\ref{sym}) extend to $\textbf{M}_g^*$. Using these sequences, we can compute the class $c_1(\G_{i, 2}-\H_{i, 2})\in \mathrm{Pic}(\textbf{M}_g^*)=\mathrm{Pic}(\rem_g)$. One finds a formula
$$c_1(\G_{i, 2}-\H_{i, 2})=\mathfrak{a} \lambda-\mathfrak{b}_0 \delta_0-\cdots-\mathfrak{b}_{[g/2]}\delta_{[g/2]},$$
where $\mathfrak{b}_j\geq \mathfrak{b}_0$ for $j\geq 1$. It follows that $$s([\overline{\cZ}_{g, i}]^{virt})=\frac{\mathfrak{a}}{\mathfrak{b}_0}=6+\frac{12}{g+1}.$$ This finishes the proof of Theorem \ref{hme} provided we can show that $\cZ_{g, i}$ is an "honest" divisor on $\cM_{2i+3}$, that is, $\gamma_{i, 2}$ is generically nondegenerate. This is the subject of Voisin's theorem \cite{V2} which gives an affirmative answer to Green's Conjecture for generic curves of odd genus (see e.g. \cite{GL} for more background):
\begin{theorem}
For a general curve $[C]\in \cM_{2i+3}$ we have the vanishing $K_{i, 2}(C, K_C)= 0$. It follows that $\cZ_{2i+3, i}$ is a divisor on $\cM_{2i+3}$.
\end{theorem}
\begin{remark} For $g=23$ Theorem \ref{clasa} shows that $s(\overline{\cZ}_{23, 10})=s(K_{\mm_{23}})=13/2$. This
implies that $\kappa(\mm_{23})\geq 0$, in particular $\mm_{23}$ is not uniruled. A finer analysis using Brill-Noether divisors on $\mm_{23}$ proves the stronger inequality $\kappa(\mm_{23})\geq 2$, cf. \cite{F1}.
\end{remark}

We finish this section by briefly discussing the proof of Theorem \ref{hme} in even genus. This is achieved in
\cite{EH3} and it relies on the calculation of class of the Gieseker-Petri divisor on $\mm_g$. We fix integers
$r, s\geq 1$ and set $g:=s(r+1)$ and $d:=r(s+1)$. Note that $\rho(g, r, d)=0$ and every even genus appears in this way. A general curve $[C]\in \cM_g$ has a finite number of linear series $L\in W^r_d(C)$ and for each of them, the multiplication map
$$\mu_0(L): H^0(C, L)\otimes H^0(C, K_C\otimes L^{\vee})\rightarrow H^0(C, K_C)$$ is an isomorphism. We define the Gieseker-Petri locus
$$\mathcal{GP}^r_{g, d}:=\{[C]\in \cM_g: \exists L\in W^r_d(C) \mbox{ such that }
\mu_0(L) \mbox{ is not injective }\}.$$
The following result is proved in \cite{F3} Theorem 1.6. The case $s=2, g=2r+2$, which is the most important and the one used in the proof of Theorem \ref{hme}, has been originally settled in \cite{EH3}. The proof given in \cite{F3} which uses the techniques of Koszul cohomology, is however substantially shorter.
\begin{theorem}\label{gp}
For $d=rs+r$ and $g=rs+s$, the locus $\mathcal{GP}_{g, d}^r$ has at least one divisorial component. The slope of the divisorial part of its compactification $\overline{\mathcal{GP}}_{g, d}^r$
in $\mm_g$ is given by the formula:
$$s(\overline{\mathcal{GP}}_{g, d}^r)=6+ \frac{12}{g+1}+\frac{6(s+r+1)(rs+s-2)(rs+s-1)}{s(s+1)(r+1)(r+2)(rs+s+4)(rs+s+1)}.$$
\end{theorem}

\section{The locus of $K3$ sections in the moduli space}

To extend Theorem \ref{hme} and show that $\mm_g$ is of general for a genus $g\leq 23$, one needs to construct effective divisors $D\in
\mbox{Eff}(\mm_g)$ of slope
$$s(D)<\frac{13}{2}\leq 6+\frac{12}{g+1}.$$
One is lead to consider geometric conditions for curves $[C]\in \cM_g$, which are divisorial in moduli but are satisfied by
all curves lying on $K3$ surfaces. Thus it makes sense to study more systematically the geometry of curves of arbitrary genus on
$K3$ surfaces.

Let $S$ be a $K3$ surface and $C\subset S$ a smooth curve of genus $g$.
We choose a linear series $A\in W^r_d(C)$ with $d\leq g-1$, satisfying the following properties:
\begin{itemize}
\item Both linear series $A\in W^r_d(C)$ and $K_C\otimes A^{\vee}\in W_{2g-2-d}^{g-d+r-1}(C)$
are base point free.
\item Both multiplication maps
$$H^0(C, A)\otimes H^0(C, K_C)\rightarrow H^0(C, A\otimes K_C)$$
and
$$H^0(C, A)\otimes H^0(C, K_C^{\otimes 2}\otimes A^{\vee})\rightarrow
H^0(C, K_C^{\otimes 2})$$
are surjective.
\end{itemize}
We recall that the Lazarsfeld bundle $M_A$ on $C$ comes from the exact sequence
$$ 0\rightarrow M_A\rightarrow H^0(A)\otimes \OO_C\stackrel{ev_C}\rightarrow A\rightarrow
0$$ and we set $Q_A:=M_A^{\vee}$, hence $\mbox{rank}(Q_A)=r$ and $\mbox{det}(Q_A)=K_C$.
Following an idea due to Mukai \cite{Mu3}, we show that $C$ possesses many higher rank vector bundles
with unexpectedly many global sections. These bundles are restrictions of vector bundles on $S$ and their
existence will ultimately single out the $K3$ locus $\K_g$ in $\cM_g$:

\begin{theorem} Given a smooth curve $C\subset S$ and $A\in W^r_d(C)$ as above, there exists a vector bundle
$E_A\in SU_C(r+1, K_C)$ sitting in an exact sequence
$$0\longrightarrow Q_A\rightarrow E_A\longrightarrow A\longrightarrow 0,$$
and satisfying the condition \ $h^0(C, E_A)=h^0(C, A)+h^0(C, K_C\otimes A^{\vee})=g-d+2r+1.$
\end{theorem}
\begin{proof} Viewing $A$ as a sheaf on $S$, we define the sheaf $\tilde{F}_A$ through
the exact sequence
$$0\longrightarrow \tilde{F}_A\longrightarrow H^0(A)\otimes \OO_S \stackrel{ev_S}\longrightarrow A\longrightarrow 0.$$
Since $A$ is a base point free line bundle, $\tilde{F}_A$ is a vector bundle on $S$. We consider the vector bundle $\tilde{E}_A:=\tilde{F}_A^{\vee}$ on $S$,
which sits in an exact sequence
\begin{equation}\label{e1}
0\longrightarrow H^0(A)^{\vee}\otimes \OO_S\longrightarrow \tilde{E}_A\longrightarrow K_C\otimes A^{\vee}\longrightarrow 0.
\end{equation}
We write down the following commutative diagram
$$\begin{array}{ccccccccc}
& \; &\;  0 & \; & 0 & \;&   \\
 & \; & \rmapdown{} &\; & \rmapdown{}\\
& \;  & H^0(A)\otimes \OO_S(-C) & \stackrel{=}\longrightarrow & H^0(A)\otimes \OO_S(-C) & \; &\;   \\
 & \; & \rmapdown{} & \; & \rmapdown{} & \; & \; & \\
   0 & \longrightarrow & \tilde{F}_A & \longrightarrow  & H^0(A)\otimes \OO_S & \longrightarrow & A & \longrightarrow &
   0 \\
  & \; &\rmapdown{} & \; &\rmapdown{} & \; &\rmapdown{=} &  \; & \\
 0 &\longrightarrow & M_A &  \longrightarrow & H^0(A)\otimes \OO_C & \longrightarrow &
A &\longrightarrow & 0 \\
 &\; & \rmapdown{} & \; & \rmapdown{}\\
 & \;&  0 & \;&  0\\
\end{array}$$
from which, if we set $F_A:=\tilde{F}_A\otimes \OO_C$ and $E_A:=\tilde{E}_A\otimes \OO_C$, we obtain the exact sequence
$$0\longrightarrow M_A\otimes K_C^{\vee}\longrightarrow H^0(A)\otimes K_C^{\vee}\longrightarrow F_A\longrightarrow M_A\longrightarrow 0$$
(use that $\mbox{Tor}_{\OO_S}^1(M_A, \OO_C)=M_A\otimes K_C^{\vee}$). Taking duals, we find the exact sequence
\begin{equation}\label{e2}
0\longrightarrow Q_A\longrightarrow E_A\longrightarrow K_C\otimes A^{\vee}\longrightarrow 0.
\end{equation}
Since $S$ is regular, from (\ref{e1}) we obtain that $h^0(S, \tilde{E}_A)=h^0(C, A)+h^0(C, K_C\otimes A^{\vee})$
while $H^0(S, \tilde{E}_A\otimes \OO_S(-C))=0$, that is, $$h^0(S, \tilde{E}_A)\leq h^0(C, E_A)\leq h^0(C, A)+h^0(C, K_C\otimes A^{\vee}).$$ This shows that the sequence (\ref{e2}) is exact on global sections and completes the proof.
\end{proof}
\begin{corollary}\label{multi}
Let $C\subset S$ and $A\in W^r_d(C)$ be as above. Then the multiplication map
$$
H^0(C, K_C\otimes A^{\vee})\otimes H^0(C, K_C\otimes M_A)\rightarrow H^0(C, K_C^{\otimes 2}\otimes A^{\vee}\otimes M_A)$$
is not surjective. In particular, for every base point free pencil $A\in W^1_d(C)$ with $d\leq g-1$, the multiplication map
$$\mathrm{Sym}^2 H^0\bigl(C, K_C\otimes A^{\vee}\bigr)\rightarrow H^0\bigl(C, K_C^{\otimes 2}\otimes A^{\otimes (-2)}\bigr)$$ is not surjective.
\end{corollary}
\begin{proof}
The existence of the bundle $E_A\in \mbox{Ext}^1(K_C\otimes A^{\vee}, Q_A)=H^0(C, K_C^{\otimes 2}\otimes A^{\vee}\otimes M_A)^{\vee}$ satisfying $h^0(C, E_A)=h^0(C, Q_A)+h^0(C, K_C\otimes A^{\vee})$ implies
that the coboundary map
$$\mbox{Ext}^1\Bigl(K_C\otimes A^{\vee}, Q_A\Bigr)\rightarrow \mbox{Hom}\Bigl(H^0(C, K_C\otimes A^{\vee}), H^1(C, Q_A)\Bigr)$$
given by $E\mapsto \delta_E$, is not injective. We finish the proof by applying Serre duality.
\end{proof}
\begin{corollary}\label{conditie2}
For $C\subset S$ and $A\in W^r_d(C)$ as above, we have that
$$h^0(C, Q_A\otimes Q_{K_C\otimes A^{\vee}})\geq h^0(C, A)h^0(C, K_C\otimes A^{\vee})+1.$$
\end{corollary}
\begin{proof}
We tensor the exact sequence $$0\longrightarrow M_{K_C\otimes A^{\vee}}\longrightarrow H^0(K_C\otimes A^{\vee})\otimes
\OO_C\rightarrow K_C\otimes A^{\vee}\longrightarrow 0$$ by the vector bundle $M_A\otimes K_C$, then apply Corollary
\ref{multi}. The conclusion follows because by assumption $H^1(C, K_C^{\otimes 2}\otimes A^{\vee}\otimes M_A)=0.$
\end{proof}

Corollary \ref{multi} can be simplified in the case of linear series of dimension $\geq 2$. For instance we
have the following characterization which will be used in Section 6:
\begin{proposition}\label{conditie3}
Given $C\subset S$ a Brill-Noether general curve and $A\in W^2_d(C)$ a complete linear series as above, the multiplication map
$$\mathrm{Sym}^2 H^0(C, K_C\otimes A^{\vee})\rightarrow H^0(C, K_C^{\otimes 2}\otimes A^{\otimes (-2)})$$
is not surjective.
\end{proposition}
\begin{proof} We start by choosing points $p, q\in C$ such that $A\otimes \OO_C(-p-q)\in W^1_{d-2}(C)$.
We can write the following exact sequence
$$0\longrightarrow \OO_C(p+q)\longrightarrow Q_A\longrightarrow A\otimes \OO_C(-p-q)\longrightarrow 0,$$
which we use together with  Corollary \ref{conditie2} to write the inequalities
$$h^0(C, A)\ h^0(C, K_C\otimes A^{\vee})+1\leq h^0(C, Q_A\otimes Q_{K_C\otimes A^{\vee}})\leq $$
$$\leq h^0\bigl(C, Q_{K_C\otimes A^{\vee}}\otimes \OO_C(p+q)\bigr)+h^0\bigl(C, Q_{K_C\otimes A^{\vee}}\otimes A\otimes \OO_C(-p-q)\bigr).$$
We apply the Base point free pencil trick to note that the multiplication map
$$H^0(C, K_C(-p-q))\otimes H^0(C, K_C\otimes A^{\vee})\rightarrow H^0(C, K_C^{\otimes 2}\otimes A^{\vee}(-p-q))$$
is surjective, hence $h^0(C, Q_{K_C\otimes A^{\vee}}(p+q))=h^0(C, K_C\otimes A^{\vee})$. Then one must have
$$h^0\bigl(C, Q_{K_C\otimes A^{\vee}}\otimes A(-p-q)\bigr)>2h^0(C, K_C\otimes A^{\vee}),$$ which implies that the
multiplication map
$$H^0(C, K_C\otimes A^{\vee})\otimes H^0(C, K_C\otimes A^{\vee}(p+q))\rightarrow H^0(C, K_C^{\otimes 2}\otimes A^{\otimes (-2)}(p+q))$$
is not surjective. Since $h^0(C, K_C\otimes A^{\vee}(p+q))=h^0(C, K_C\otimes A^{\vee})+1$, this implies that the map
$$\mathrm{Sym}^2 H^0(C, K_C\otimes A^{\vee})\rightarrow H^0(C, K_C^{\otimes 2}\otimes A^{\otimes (-2)})$$ is not surjective either.
\end{proof}

\begin{example}
As an illustration, a general curve $[C]\in \cM_{21}$ carries a
finite number of linear series $A\in W^2_{16}(C)$ and $C\stackrel{|K_C\otimes A^{\vee}|}\hookrightarrow \PP^6$ is an embedding for all $A\in W^2_{16}(C)$. The locus
$$\cZ_{21}:=\{[C]\in \cM_{21}: \exists A\in W^2_{16}(C) \mbox{ with } \mathrm{Sym}^2  H^0(C, K_C\otimes A^{\vee})\ncong H^0(C, K_C^{\otimes 2}\otimes A^{\otimes (-2)})\}$$
contains the locus $\mathcal{K}_{21}$ of sections of $K3$ surfaces. Since $$\mbox{rank} \mbox{ Sym}^2 H^0(C, K_C\otimes A^{\vee})=\mbox{rank } H^0(C, K_C^{\otimes 2}\otimes A^{\otimes (-2)}),$$ clearly $\cZ_{21}$ is a virtual divisor on $\cM_{21}$. In fact  $\overline{\cZ}_{21}$ is an "honest" divisor on $\mm_{21}$
of slope $s(\overline{\cZ}_{21})<6+12/22$ (cf. \cite{F3}, \cite{Kh}). Unfortunately, $s(\overline{\cZ}_{21})>6.5$, so one cannot conclude that
$\mm_{21}$ is of general type.
\end{example}

To summarize, the existence of the vector bundles $E_A$ shows that curves $C$ on $K3$ surfaces carry line bundles
of the form $K_C\otimes A^{\vee}$  having very special geometric properties (Corollary \ref{conditie2}). The vector bundles $E_A$ are produced starting from any linear series $A\in W^r_d(C)$ satisfying suitable genericity condition. This leads to the construction of Koszul
divisors on $\mm_g$ as being push-forwards of degeneracy loci defined on stacks $\widetilde{\mathfrak G}^r_d$ of limit linear series, cf. \cite{F2}, \cite{F3}.

\subsection{Koszul divisors on $\mm_g$. }
We can rewrite Corollary \ref{multi} in terms of
Koszul cohomology groups. A curve $[C]\in \K_g$ enjoys the property that $K_{0, 2}(C, K_C\otimes A^{\vee})\neq 0$
for every pencil $A\in W^1_d(C)$ with $d\leq g-1$ such that $K_C\otimes A^{\vee}$ is globally generated.
This suggests an obvious ways of constructing geometric divisors on $\cM_g$ which contain the $K3$ locus $\mathcal{K}_g$
by looking at the higher (rather than $0$-th order) Koszul cohomology groups $K_{i, 2}(C, K_C\otimes A^{\vee})$. From a technical point of view the simplest case is when one considers syzygies of linear series residual to a pencil of minimal degree in the case when
the general curve $[C]\in \cM_g$ has a finite number of such pencils.  The situation when the Brill-Noether number is positive will be considered in the forthcoming paper \cite{F4}. A special case of that new construction can be found in Section 7 of this paper.

 We fix an integer $i\geq 0$ and set
$$g:=6i+10,\ d:=9i+12, \ \mbox{ and } r:=3i+4$$
hence $\rho(g, r, d)=0$. We consider the open substack $\textbf{M}_{g}^0\subset \textbf{M}_g$ consisting of curves $[C]\in \cM_{g}$ such that $W_{d-1}^r(C)=\emptyset$ and $W_d^{r+1}(C)=\emptyset$. Note that for a curve $[C]\in \cM_{g}^0$, each $L\in W^{3i+4}_{9i+12}(C)$ is complete and base point free. From Riemann-Roch, the residual linear series $K_C\otimes L^{\vee}\in W^1_{3i+6}(C)$ is a pencil of minimal degree. We would like to study  the locus of curves $[C]\in \cM_g$ carrying a linear series $L\in W^r_d(C)$ with extra syzygies of order $i$. Our numerical choices for $g, r$ and $d$
imply that this locus is a (virtual) divisor on $\mm_g$. Whenever it is a divisor, it is guaranteed to contain $\mathcal{K}_g$. The next theorem comes from \cite{F2}:

\begin{theorem}
There exists a partial compactification $\textbf{M}_g^0\subset  \pem_g\subset \rem_g$ of the stack of smooth curves with $\mathrm{codim}(\rem_g-\pem_g)\geq 2$, such that if
$$\sigma:\widetilde{\mathfrak G}^{1}_{3i+6} \rightarrow \pem_{g}$$ denotes the
stack of limit linear series, then there exist vector bundles $\cA$ and $\cB$ of the same rank
together with a vector bundle morphism
$\phi_i:\cA\rightarrow \cB$ over $\widetilde{\mathfrak G}^{1}_{3i+6}$ such
that the degeneracy locus of $\phi_i$ over $\sigma^{-1}(\textbf{M}_g^0)$ equals
$$\cZ_{g, i}:=\{[C, A]\in \mathfrak G^1_{3i+6}: K_{i, 2}(C, K_C\otimes A^{\vee})\neq 0\}.$$ The slope of the virtual class of $\zz_{g, i}$ is equal to
$$s\bigl([\zz_{g, i}]^{virt}\bigr)=s\bigl(\sigma_*c_1(\cB-\cA)\bigr)=\frac{3(4i+7)(6i^2+19i+12)}{(i+2)(12i^2+31i+18)}<6+\frac{12}{g+1}.$$
\end{theorem}

The question of generic non-degeneracy  of the morphism $\phi_i$ is addressed in \cite{F2}. It is proved that $\phi_i$ is  generically non-degenerate for $i=0, 1, 2$. In particular, the locus $\overline{\cZ}_{22, 2}$ is an effective divisor
on $\mm_{22}$ of slope $s(\overline{\cZ}_{22, 2})=1665/256=6.5032...$. This barely fails to make $\mm_{22}$ of general type!

 It is conjectured in \cite{F2} that $\overline{\cZ}_{g, i}$ is an actual divisor on $\mm_{6i+10}$ for all $i\geq 0$.
To show that $\cZ_{22, 2}$ is a divisor on $\cM_{22}$ (rather than the entire space $\cM_{22}$), we use that (i) the Hurwitz stack $\mathfrak{G}^{1}_{12}$ is irreducible and (ii) one can find a smooth  embedded genus $22$ curve $C\stackrel{\mathfrak g^{10}_{30}}\hookrightarrow \PP^{10}$ of genus $22$, such that $K_{2, 2}(C, \mathfrak g_{30}^{10})=0$. In other words, $C\subset \PP^{10}$ is cut out by quadrics and all the syzygies among the quadrics are linear.

Because $\mathfrak G^{1}_{12}$ is irreducible,  this implies that if $[C]\in \cM_{22}$ is a general curve, then $K_{2, 2}(C, K_C\otimes A^{\vee})=0$, for all $A\in W^1_{12}(C)$.
The irreducibility of the Hurwitz stack $\mathfrak G^1_{12}$ makes it possible to derive information about all $\mathfrak g^1_{12}$'s on a general curve, even though we can only see one $\mathfrak g^1_{12}$ at a time. This trick (which has been used again in \cite{F3} to prove the Maximal Rank Conjecture),  only works in the case $\rho(g, r, d)=0$. Proving transversality statements for Koszul divisors in the case $\rho(g, r, d)\geq 1$ requires different ideas.

\section{The Kodaira dimension of $\mm_{22}$}

In this section we outline the calculation of the class $[\overline{\mathfrak{D}}_{22}]$
of an effective divisor on $\mm_{22}$ of slope less than $13/2$.
Complete details of a more general construction (of which Theorem \ref{m22} is a
particular case) will appear in \cite{F4}. Precisely, we shall present in \cite{F4} a way of computing the class
of all Koszul divisors on $\mm_g$ defined in terms of linear series $\mathfrak g^r _d$ in the case
$\rho(g, r, d)=1$. (The case $\rho(g, r, d)=0$ has been dealt with in \cite{F3}). Specializing $(g, r, d)=(22, 6, 25)$ we obtain our result on the Kodaira dimension
of $\mm_{22}$.

\begin{theorem}\label{m22}
The following locus of smooth curves of genus $22$ \vskip 3pt

$\mathfrak{D}_{22}:=\{[C]\in \cM_{22}: \exists L\in W^6_{25}(C)
\mbox{ with } \rm{Sym}$$^2 H^0(C, L)\rightarrow H^0(C, L^{\otimes
2}) \mbox{ not injective} \}$ \vskip 3pt \noindent is a divisor on
$\cM_{22}$. The class of its compactification on $\mm_{22}$ is
given by the formula:
$$\overline{\mathfrak{D}}_{22}\equiv
132822768\Bigl(\frac{17121}{2636}\lambda-\delta_0-\frac{14511}{2636}\delta_0-\sum_{j=2}^{11}
b_j\delta_j\Bigr),$$ where $b_j>1$ for $2\leq j\leq 11$. It follows
that $s(\overline{\mathfrak{D}}_{22})=17121/2636=6.49506\ldots$,
therefore $\mm_{22}$ is of general type.
\end{theorem}

We discuss the calculation of the class of $\overline{\mathcal{D}}_{22}$ viewed as
a virtual degeneracy locus on a partial compactification of $\mm_{22}$. The proof that
$\overline{\mathcal{D}}_{22}$ is indeed a divisor on $\mm_{22}$, that is, that for a general curve
$[C]\in \cM_{22}$ we have that $\mbox{Sym}^2 H^0(C, L)\rightarrow H^0(C, L^{\otimes 2})$ is injective for
\emph{all} $L\in W^6_{25}(C)$ will be presented in \cite{F4} as part of a more general version of the Maximal Rank Conjecture (see again \cite{F3} Theorem 1.5 for the corresponding statement when $\rho(g, r, d)=0$).

The idea is to construct two tautological vector bundles over
the Severi variety $\mathfrak G^2_{17}$ of curves $[C]\in \cM_{22}$ with a plane model $\mathfrak g^2_{17}$ and
then define the divisor $\mathfrak{D}_{22}$ as the image of the first
degeneration locus of a natural map between these bundles.

We denote by $\textbf{M}_{22}^p$ the open substack of $\textbf{M}_{22}$ consisting
of curves $[C]\in \cM_{22}$ such that $W^6_{24}(C)=\emptyset$ and
$W^7_{25}(C)=\emptyset$. Standard results in Brill-Noether theory
guarantee that $\mbox{codim}(\cM_{22}-\cM_{22}^p, \cM_{22})\geq 2$.
If $\mathfrak{Pic}^{25}_{22}$ denotes the Picard stack of degree
$25$ over $\cM_{22}^p$, then we consider the substack
$\mathfrak{G}^6_{25}\subset \mathfrak{Pic}^{25}_{22}$ parameterizing
pairs $[C, L]$ where $[C]\in \cM_{22}^p$ and $L\in W^6_{25}(C)$. We
denote by $$\sigma: \mathfrak{G}^6_{25}\rightarrow \textbf{M}_{22}^p$$ the
forgetful morphism. For a general $[C]\in \cM_{22}^p$, the fibre
$\sigma^{-1}([C])=W^6_{25}(C)$ is a smooth curve and
$\mathfrak{G}^6_{25}$ is an irreducible stack of dimension
$\mbox{dim }\mathfrak{G}^6_{25}=\mbox{dim }\cM_{22}+1$.

Let $\pi:\textbf{M}_{22, 1}^p \rightarrow \textbf{M}_{22}^p$ be the universal
curve and then $p_2:\textbf{M}_{22, 1}^p\times
_{\textbf{M}_{22}^p} \mathfrak{G}^6_{25}\rightarrow \mathfrak{G}^6_{25}$
denotes the natural projection. If $\mathcal{L}$ is a Poincar\'e bundle
over $\textbf{M}_{22, 1}^p\times_{\textbf{M}_{22}^p} \mathfrak{G}^6_{25}$, then by
Grauert's Theorem $\E:=(p_2)_*(\mathcal{L})$ and
$\F:=(p_2)_*(\mathcal{L}^{\otimes 2})$ are vector bundles over
$\mathfrak{G}^6_{25}$ with $\mbox{rank}(\E)=7$ and
$\mbox{rank}(\F)=29$. There is a natural vector bundle morphism over $\mathfrak G^6_{25}$
$$\phi:\mbox{Sym}^{2}(\E)\rightarrow \F$$ and we denote by
$\cU_{22}\subset \mathfrak{G}^6_{25}$ its first degeneracy locus. We
set $\mathfrak{D}_{22}:=\sigma_*(\cU_{22})$ and clearly $\cU_{22}$
has expected codimension $2$ inside $\mathfrak{G}^6_{25}$ hence
$\mathfrak{D}_{22}$ is a virtual divisor on $\cM_{22}^p$.

Using Proposition \ref{conditie3}, we are guaranteed that $\mathfrak{D}_{22}$ contains the $K3$ locus $\mathcal{K}_{22}$, in particular it is a good candidate for a divisor on $\mm_{22}$ of exceptionally small slope.
We shall extend the vector bundles $\E$ and $\F$ over a partial
compactification of $\mathfrak{G}^6_{25}$. We denote by
$\Delta_1^p\subset \Delta_1^0\subset \mm_{g}$ the locus of curves
$[C\cup_y E]$, where $E$ is an arbitrary elliptic curve, $[C]\in
\cM_{g-1}$ is a Brill-Noether general curve and $y\in
C$ is an arbitrary point. We also denote by $\Delta_0^p\subset
\Delta_0^0\subset \mm_{g}$ the locus consisting of curves $[C_{yq}]\in \Delta_0^0$, where $[C, q]\in \cM_{g-1, 1}$ is Brill-Noether general and $y\in
C$ is arbitrary, as well as their degenerations $[C\cup_q E_{\infty}]$ where
$E_{\infty}$ is a rational nodal curve (that is, $j(E_{\infty})=\infty$).
Once we set
$$\ttem_g:=\tem_g \cup \Delta_0^p\cup \Delta_1^p\subset \pem_g,$$
we can extend the map $\sigma$ to a proper morphism
$\sigma:\widetilde{\mathfrak{G}}^6_{25}\rightarrow
\ttem_{22}$ from the stack
$\widetilde{\mathfrak{G}}^6_{25}$ of limit linear
series $\mathfrak g^6_{25}$ over the partial compactification  $\ttem_{22}$ of $\textbf{M}_{22}$.

Like in to \cite{F2}, \cite{F3} or in Section 5 of this paper, we intersect the (virtual) divisor
$\overline{\mathfrak{D}}_{22}$ with the test curves $C^0\subset \Delta_0^p$ and $C^1\subset \Delta_1^p$
obtained from  a general pointed curve $[C, q]\in \cM_{21, 1}$
and a general elliptic curve $[E, y]\in \cM_{1, 1}$. We explicitly describe the pull-back $2$-cycles under $\sigma$ of the test curves $C^0$ and $C^1$:

\begin{proposition}\label{limitlin1}
Fix general curves $[C]\in \cM_{21}$ and $[E, y]\in \cM_{1, 1}$ and
consider the associated test curve $C^1\subset \Delta_1\subset
\mm_{22}$. Then we have the following equality of $2$-cycles in
$\widetilde{\mathfrak{G}}_{25}^6$:
$$\sigma^*(C^1)=X+ X_1\times X_2+ \Gamma_0\times  Z_0+ n_1\cdot Z_1+n_2\cdot Z_2+n_3\cdot Z_3,$$
where
$$X:=\{(y, L)\in C\times W^6_{25}(C): h^0(C, L\otimes
\OO_C(-2y))=6\},$$ $$ X_1:=\{(y, L)\in C\times W^6_{25}(C):
a^{L}(y)=(0, 2, 3, 4, 5, 6, 8)\}, $$ $$ X_2:=\{l_E\in G^6_8(E):
a_1^{l_E}(y)\geq 2, a_6^{l_E}(y)=8\}\cong
\PP\Bigl(\frac{H^0(\OO_E(8y))}{H^0(\OO_E(6y))}\Bigr)$$
$$\Gamma_0:=\{(y, A\otimes \OO_C(y)): y\in C, A\in W^6_{24}(C)\},\mbox{
} Z_0=G^6_7(E)\cong E,$$ $$ Z_1:=\{l_E\in G^6_9(E): a_1^{l_E}(y)\geq
3, a_6^{l_E}(y)=9\}\cong
\PP\Bigl(\frac{H^0(\OO_E(9y))}{H^0(\OO_E(6y))}\Bigr),$$
$$Z_2:=\{l_E\in G^6_8(E):a_2^{l_E}(y)\geq 3, a_6^{l_E}(y)=8\}\cong
\PP\Bigl(\frac{H^0(\OO_E(8y))}{H^0(\OO_E(5y))}\Bigr),$$
$$Z_3:=\{l_E\in G^6_8(E): a^{l_E}(y)\geq (0, 2, 3, 4, 5, 6, 7)\} \cong \bigcup_{z\in E} \PP\Bigl(\frac{H^0(\OO_E(7y+z))}
{H^0(\OO_E(5y+z))}\Bigr), $$ where the constants $n_1, n_2$
 and $n_3$  are explicitly known positive integers.
\end{proposition}
\begin{remark} The constants $n_i, 1\leq i\leq 3$ have the following enumerative interpretation. First  $n_1$ is the number of linear series $L\in W^6_{25}(C)$ such that
there exists an unspecified point $y\in C$ with $a^L(y)=(0, 2,3, 4,
5, 6, 9)$. Similarly, $n_2$ is the number of those $L\in
W^6_{25}(C)$ for which there exists $y\in C$ with $a^{L}(y)=(0, 2,
3, 4, 5, 7, 8)$. Finally $n_3$ is the number of points $y\in C$ such
that there exists $L\in W^6_{24}(C)$ which is ramified at $y$. If $n_0$ is the number of $\mathfrak g^6_{24}$'s on $C$, then $\Gamma_{0}$ consists of $n_0$ disjoint copies of the curve $C$.
\end{remark}
Before describing $\sigma^*(C^0)$, we set some more notation. For a general pointed curve $[C, q]\in \cM_{21, 1}$ we denote by $Y$ the
surface
$$Y:=\{(y, L)\in C\times W^6_{25}(C): h^0(C, L\otimes
\OO_C(-y-q))=6\}$$ and by $\pi_1:Y \rightarrow C$ the first
projection. Inside $Y$ we consider two curves corresponding to
$\mathfrak g^6_{25}$'s with a base point at $q$:
$$\Gamma_1:=\{(y, A\otimes \OO_C(y)): y\in C, A\in W^6_{24}(C)\}
\ \mbox{  and}$$
$$\Gamma_2:=\{(y, A\otimes \OO_C(q)): y\in C, A\in
W^6_{24}(C)\}$$ intersecting transversally in $n_0=\#\bigl(W^6_{24}(C)\bigr)$
points. Note that since $[C]\in \cM_{21}$ is Brill-Noether general,
$W^6_{24}(C)$ is a reduced $0$-dimensional scheme consisting of
$n_0$ very ample (in particular, base point free) $\mathfrak
g^6_{24}$'s. We denote by $Y'$ the blow-up of $Y$ at these $n_0$
points and at the points $(q, B)\in Y$ where $B\in W^6_{25}(C)$ is a
linear series with the property that $h^0(C, B\otimes
\OO_C(-8q))\geq 1$. We denote by $E_A, E_B\subset Y'$ the
exceptional divisors corresponding to $(q, A\otimes \OO_C(q))$ and
$(q, B)$ respectively, by $\epsilon: Y'\rightarrow Y$ the projection
and by $\widetilde{\Gamma}_1, \widetilde{\Gamma}_2\subset Y'$ the
strict transforms of $\Gamma_1$ and $\Gamma_2$ respectively.

\begin{proposition}\label{limitlin0}
Fix a general curve $[C, q]\in \cM_{21, 1}$ and consider the
associated test curve $C^0\subset \Delta_0\subset \mm_{22}$. Then we
have the following equality of $2$-cycles in
$\widetilde{\mathfrak{G}}_{25}^6$:
$$\sigma^*(C^0)=Y'/\widetilde{\Gamma}_1\cong \widetilde{\Gamma}_2,$$
that is, $\sigma^*(C^0)$ can be naturally identified with the
surface obtained from $Y'$ by identifying the disjoint curves
$\widetilde{\Gamma}_1$ and $\widetilde{\Gamma}_2$ over each pair
$(y, A)\in C\times W^6_{24}(C)$.
\end{proposition}
\begin{proof}
We fix a point $y\in C-\{q\}$, denote by $[C_{yq}:=C/ y\sim q]\in
\Delta_0^p\subset \mm_{22}$ and by $\nu:C\rightarrow C_{yq}$ the normalization map. We
describe the variety $\overline{W}^6_{25}(C_{yq})\subset
\overline{\mbox{Pic}}^{25}(C_{yq})$ of torsion-free sheaves $L$ on
the $1$-nodal curve $C_{yq}$, with $\mbox{deg}(L)=25$ and $h^0(C_{yq}, L)\geq 7$.

If $L\in
W^6_{25}(C_{yq})\subset \overline{W}^6_{25}(C_{yq})$, that is, $L$ is a locally free sheaf, then $L$ is
completely determined by $\nu^*(L)\in W^6_{25}(C)$ which has the property that
$h^0(C, \nu^*L\otimes \OO_C(-y-q))=6$. However, the line bundles of
type $A\otimes \OO_C(y)$ or $A\otimes \OO_C(q)$ with $A\in
W^6_{24}(C)$, do not appear in this association even though they
have this property. In fact, they correspond to the situation when
$L\in \overline{W}_{25}^6(C_{yq})$ is not locally free, in which case
necessarily one has that $L=\nu_*(A)$, for some $A\in W^6_{24}(C)$. Thus $Y\cap
\pi_1^{-1}(y)$ is the partial normalization of
$\overline{W}_{25}^6(C_{yq})$ at the $n_0$ points of the form
$\nu_*(A)$ with $A\in W^6_{24}(C)$. A special analysis is required
when $y=q$, that is, when $C_y^0$ degenerates to $C\cup _q
E_{\infty}$, where $E_{\infty}$ is a rational nodal cubic. If
$\{l_C, l_{E_{\infty}}\}\in \sigma^{-1}([C\cup_{q} E_{\infty}])$,
then an argument along the lines of Theorem \ref{limitlin1} shows
that $\rho(l_C, q)\geq 0$ and $\rho(l_{E_{\infty}}, q)\leq 1$. Then
either $l_C$ has a base point at $q$ and then the underlying line
bundle of $l_C$ is of type $A\otimes \OO_C(q)$ while
$l_{E_{\infty}}(-18q)\in \overline{W}_7^6(E_{\infty})$, or else,
$a^{l_C}(q)=(0, 2, 3, 4, 5, 6, 8)$ and then $l_{E_{\infty}}(-17q)\in
\PP\bigl(H^0(E_{\infty}(8q))/H^0(E_{\infty}(6q))\bigr)\cong E_B$,
where $B\in W^6_{25}(C)$ is the underlying line bundle of $l_C$.
\end{proof}

We extend the vector bundles $\E$ and $\F$ over the
stack  $\widetilde{\mathfrak G}^6_{25}$ of limit linear series. The proof of the following result proceeds along the lines of the proof of
Proposition 3.9 in \cite{F2}:
\begin{proposition}
There exist two vector bundles $\E$ and $\F$ defined over
$\widetilde{\mathfrak G}^6_{25}$ with $\rm{rank}$$(\E)=7$ and
$\mathrm{rank}(\F)=29$ together with a vector bundle morphism $\phi:\mathrm{Sym}^2(\E)\rightarrow \F$,
such that the following statements hold:
\begin{itemize}
\item For $(C, L)\in \mathfrak{G}^6_{25}$, with $[C]\in \cM_{22}^p$, we have that
$\E(L)=H^0(C, L)$ and $\F(L)=H^0(C, L^{\otimes 2}).$
\item For $t=(C\cup_y E, l_C, l_E)\in \sigma^{-1}(\Delta_1^p)$,
where $g(C)=21, g(E)=1$ and $l_C=|L_C|$ is such that  $L_C\in
W^6_{25}(C)$ has a cusp at $y\in C$, then $\E(t)=H^0(C, L_C)$ and
$$\F(t)=H^0(C, L_C^{\otimes 2}(-2y))\oplus \mathbb C\cdot u^2,$$ where
$u\in H^0(C, L_C)$ is any section such that $\rm{ord}$$_y(u)=0$.
If $L_C$ has a base point at $y$, then $$\E(t)=H^0(C, L_C)=H^0(C, L_C\otimes \OO_C(-y))$$ and the image of a natural
map $\F(t)\rightarrow H^0(C, L_C^{\otimes 2})$ is the subspace $H^0(C, L_C^{\otimes 2}\otimes \OO_C(-2y))$.
\item Fix $t=(C_{yq}:=C/y\sim q, L)\in \sigma^{-1}(\Delta_0^p)$, with $q,
y\in C$ and $L\in \overline{W}^6_{25}(C_{yq})$  such that $h^0(C,
\nu^*L\otimes \OO_C(-y-q))=6$, where $\nu:C\rightarrow C_{yq}$ is the
normalization map.

In the case when $L$ is locally free we have that
$$\E(t)=H^0(C, \nu^*L)\ \mbox{ and }$$
$$ \F(t)=H^0(C, \nu^*L^{\otimes
2}\otimes \OO_C(-y-q))\oplus \mathbb C\cdot u^2,$$ where $u\in
H^0(C, \nu^*L)$ is any section not vanishing at $y$ and $q$. In the
case when $L$ is not locally free, that is,  $L\in
\overline{W}_{25}^6(C_{yq})-W_{25}^6(C_{yq})$, then $L=\nu_*(A)$,
where $A\in W^6_{24}(C)$ and the image of the natural map
$\F(t)\rightarrow H^0(C, \nu^*L^{\otimes 2})$ is the subspace
$H^0(C, A^{\otimes 2})$.
\end{itemize}
\end{proposition}

We determine the cohomology classes of the surfaces $X$ and $Y$
introduced in Propositions \ref{limitlin1} and \ref{limitlin0} respectively.
Our result are expressible in terms of standard cohomology classes
on Jacobians (cf. \cite{ACGH}, \cite{F5}), which we now recall.
If $[C]\in \cM_g$ is  a curve satisfying the Brill-Noether theorem,  we denote
by $\P$ a Poincar\'e bundle on $C\times \mbox{Pic}^d(C)$ and by
$$\pi_1:C\times \mbox{Pic}^d(C)\rightarrow C\ \mbox{ and } \ \pi_2:C\times
\mbox{Pic}^d(C)\rightarrow \mbox{Pic}^d(C)$$ the projections. We
define the cohomology class $\eta=\pi_1^*([point])\in H^2(C\times
\mbox{Pic}^d(C))$, and if $\delta_1,\ldots, \delta_{2g}\in H^1(C,
\mathbb Z)\cong H^1(\mbox{Pic}^d(C), \mathbb Z)$ is a symplectic
basis, then we set
$$\gamma:=-\sum_{\alpha=1}^g
\Bigl(\pi_1^*(\delta_{\alpha})\pi_2^*(\delta_{g+\alpha})-\pi_1^*(\delta_{g+\alpha})\pi_2^*(\delta_
{\alpha})\Bigr).$$ We have the formula $c_1(\P)=d\eta+\gamma,$
corresponding to the Hodge decomposition of $c_1(\P)$. We also
record that $\gamma^3=\gamma \eta=0$, $\eta^2=0$ and
$\gamma^2=-2\eta \pi_2^*(\theta)$. On $W^r_d(C)$  we have the
tautological rank $r+1$ vector bundle
$\mathcal{M}:=(\pi_2)_{*}(\mathcal{P}_{| C\times W^r_d(C)})$. The
Chern numbers of $\mathcal{M}$ can be computed using the Harris-Tu formula.
By repeatedly applying it, we get all intersection numbers on
$W^r_d(C)$ which we need:
\begin{lemma}\label{vandermonde}
If $[C]\in \cM_{21}$ is Brill-Noether general and
$c_i:=c_i(\mathcal{\cM}^{\vee})$ are the Chern classes of the dual
of the tautological bundle on $W^2_{17}(C)$, we have the following
identities in $H^*(W^2_{17}(C), \mathbb Z)$:
$$ [W^2_{17}(C)]=\frac{\theta^{18}}{73156608000}.$$
$$ x_1\cdot \xi=\frac{\theta^{19}\cdot \xi}{219469824000},$$
$$ \ x_2\cdot \xi=x_3
\cdot \xi=0, \mbox{ for any } \xi\in H^4(\rm{Pic}^{21}(C)).$$
$$ x_1x_2\cdot \xi=\frac{\theta^{20}}{1755758592000}\cdot \xi, $$
$$\ x_1x_3\cdot \xi=x_2x_3\cdot \xi=0,
\mbox{ for any } \xi\in H^2(\rm{Pic}^{21}(C)),$$
$$ x_1^2 \cdot \xi=\frac{\theta^{20}}{1097349120000}\cdot \xi,$$
$$  x_2^2 \cdot \xi=-x_1x_2 \cdot \xi, \ x_3^2\cdot
\xi=0, \mbox{ for any } \xi \in H^2(\rm{Pic}^{21}(C)),$$
$$ x_1^3=\frac{\theta^{21}}{7242504192000}, \ x_2^3=-\frac{t^{21}}{6584094720000},$$
$$\ x_3^3=x_1x_2x_3=\frac{\theta^{21}}
{36870930432000},$$
$$ x_1^2x_2=-x_2^3,\  x_1x_2^2=x_1^2x_3=x_2x_3^2=0,\ x_1x_3^2=x_2^2 x_3=-x_1x_2x_3. $$
\end{lemma}The next calculation is a particular case
of  \cite{F5} Proposition 2.7:
\begin{proposition}\label{xy}
Let $[C]\in \cM_{21}$ be a Brill-Noether general curve and $q\in C$
a general point. If $\cM$ denotes the tautological rank $3$ vector
bundle over $W^2_{17}(C)$ and $c_i:=c_i(\cM^{\vee})$, then one has
the following relations:
\begin{enumerate}
\item
$[X]=\pi_2^*(c_2)-6\eta \theta+(74\eta+2\gamma) \pi_2^*(c_1) \in
H^4(C\times W^2_{17}(C))$.
\item
$[Y]=\pi_2^*(c_2)-2\eta \theta+ (16\eta+\gamma) \pi_2^*(c_1) \in
H^4(C\times W^2_{17}(C))$.
\end{enumerate}
\end{proposition}
\begin{proof}
By Riemann-Roch, if $(y, L)\in X$, then the line bundle
$M:=K_C\otimes L^{\vee}\otimes \OO_C(2y)\in W^2_{17}(C)$ has a cusp
at $y$. We realize $X$ as the degeneracy locus of a vector
bundle map over $C\times W^2_{17}(C)$. For each pair $(y, M)\in C\times
W^2_{17}(C),$ there is a natural map $$H^0(C, M\otimes
\OO_{2y})^{\vee}\rightarrow H^0(C, M)^{\vee}$$ which globalizes to a
vector bundle morphism $\zeta: J_1(\mathcal{P})^{\vee} \rightarrow
\pi_2^*(\cM)^{\vee}$ over $C\times W^2_{17}(C)$ (Note that
$W^2_{17}(C)$ is a smooth $3$-fold). Then we have the identification
$X=Z_1(\zeta)$ and the Thom-Porteous formula gives that
$[X]=c_2\bigl(\pi^*_2(\cM)- J_1(\mathcal{P}^{\vee})\bigr).$ From the
usual exact sequence over $C\times \mbox{Pic}^{17}(C)$
$$
0\longrightarrow \pi_1^*(K_C)\otimes \mathcal{P} \longrightarrow
J_1(\mathcal{P}) \longrightarrow \mathcal{P} \longrightarrow 0, $$
we can compute the total Chern class of the jet bundle
$$c_t(J_1(\mathcal{P})^{\vee})^{-1}=\Bigl(\sum_{j\geq
0}(17\eta+\gamma)^j\Bigr)\cdot \Bigl(\sum_{j\geq
0}(57\eta+\gamma)^j\Bigr)=1-6\eta \theta +74\eta +2\gamma, $$ which
quickly leads to the formula for $[X]$. To compute $[Y]$ we proceed
in a similar way. We denote by $p_1, p_2:C\times C\times
\mbox{Pic}^{17}(C)\rightarrow C\times \mbox{Pic}^{17}(C)$ the two
projections, by $\Delta\subset C\times C\times \mbox{Pic}^{17}(C)$
the diagonal  and we set $\Gamma_q:=\{q\}\times \mbox{Pic}^{17}(C)$.
We introduce the rank $2$ vector bundle
$\cB:=(p_1)_*\bigl(p_2^*(\mathcal{P})\otimes
\OO_{\Delta+p_2^*(\Gamma_q)}\bigr)$ defined over $C\times
W^2_{17}(C)$ and we note that there is a bundle morphism $\chi:
\cB^{\vee}\rightarrow (\pi_2)^*(\cM)^{\vee}$ such that
$Y=Z_1(\chi)$. Since we also have that
$$c_t(\cB^{\vee})^{-1}=\bigl(1+(17\eta+\gamma)+(17\eta+\gamma)^2+\cdots\bigr)(1-\eta),$$
we immediately obtained the desired expression for $[Y]$.
\end{proof}
The next results are simple applications of Grothendieck-Riemann-Roch for
the projection morphism $p_2:C\times C\times
\mathrm{Pic}^{17}(C)\rightarrow C \times \mathrm{Pic}^{17}(C)$:

\begin{proposition}\label{a121}
Let $[C]\in \cM_{21}$ and denote by $p_1, p_2:C\times C\times
\rm{Pic}$$^{17}(C)\rightarrow C \times \rm{Pic}$$^{17}(C)$ the
natural projections. We denote by $\cA_2$ the vector bundle on
$C\times \rm{Pic}$$^{17}(C)$ with fibre at each point $\cA_2(y,
M)=H^0(C, K_C^{\otimes 2}\otimes M^{\otimes (-2)}\otimes
\OO_C(2y))$. We have the following formulas:
$$ c_1(\cA_2)=-4\theta-4\gamma-28\eta \ \mbox{ and }
c_2(\cA_2)=8\theta^2+104 \eta \theta+16 \gamma \theta.$$
\end{proposition}

\begin{proposition}\label{a120}
Let $[C, q]\in \cM_{21, 1}$ be a general pointed curve an we denote
by $\cB_2$ the vector bundle on $C\times \rm{Pic}$$^{17}(C)$ having
fibre $\cB_2(y, M)=H^0\bigl(C, K_C^{\otimes 2}\otimes M^{\otimes
(-2)}\otimes \OO_C(y+q)\bigr)$ at each point $(y, M)\in C\times
\rm{Pic}$$^{17}(C)$. Then we have that:
$$c_1(\cB_2)=-4\theta+7\eta-2\gamma \ \mbox{  and }
c_2(\cB_2)=8\theta^2-28\eta \theta+8 \theta \gamma.$$
\end{proposition}

As a first step towards computing $[\overline{\mathfrak{D}}_{22}]$
we determine the $\delta_1$ coefficient in its expression:
\begin{theorem}\label{d1}
Let $[C]\in \cM_{21}$ be Brill-Noether general and denote by
$C^1\subset \Delta_1$ the associated test curve. Then
$\sigma^*(C^1)\cdot c_2(\F- \rm{Sym}$$^2(\E))=4847375988$. It
follows that the coefficient of $\delta_1$ in the expansion of
$\overline{\mathfrak{D}}_{22}$ is equal to $b_1=731180268$.
\end{theorem}

\begin{proof} We intersect the degeneracy locus of the map
$\mbox{Sym}^2(\E)\rightarrow \F$ with the surface $\sigma^*(C^1)$
and use that the vector bundles $\E$ and $\F$ were defined by
retaining the sections of the genus $21$ aspect of each limit linear
series and dropping the information coming from the elliptic curve.
It follows that $Z_i\cdot c_2(\F-\mbox{Sym}^2(\E))=0$ for $1\leq
i\leq 3$ (since $\F$ and $\mbox{Sym}^2(\E))$ are both trivial along the
surfaces $Z_i$), and $[X_1\times X_2]\cdot
c_2(\F-\mbox{Sym}^2(\E))=0$ (because $c_2(\F-\mbox{Sym}^2(\E))_{|
X_1\times X_2}$ is in fact the pull-back of a codimension $2$ class
from the $1$-dimensional cycle $X_1$, therefore the intersection
number is $0$ for dimensional reasons). We are left with estimating
the contribution coming from $X$ and write that
$$\sigma^*(C^1)\cdot
c_2(\F-\mbox{Sym}^2(\E))=c_2(\F_{|X})-c_1(\F_{|X})c_1(\mbox{Sym}^2
\E_{| X})+c_1^2(\mbox{Sym}^2 \E_{| X})-c_2(\mbox{Sym}^2 \E_{|X}).$$
We are going to compute separately each term in the right-hand-side of this
expression.

The surface $X$ appears as the first degeneracy locus of a vector
bundle morphism $\zeta: J_1(\P)^{\vee}\rightarrow
\pi_2^*(\cM)^{\vee}$ which globalizes the maps
$$H^0(C, M\otimes
\OO_{2y})^{\vee}\rightarrow H^0(C, M)^{\vee}$$
for all $(y, M)\in C\times W^2_{17}(C)$. We denote by $U:=\mbox{Ker}(\zeta)$. In
other words, $U$ is a line bundle on $X$ with  fibre $$U(y,
M)=\frac{H^1(C, M\otimes \OO_C(-2y))^{\vee}}{H^1(C,
M)^{\vee}}=\frac{H^0(C, L)}{H^0(C, L\otimes\OO_C(-2y))}$$ over a
point $(y, M)\in X$. The Chern class of $U$ can be computed from the
Harris-Tu formula:
$$c_1(U)\cdot \xi_{| X}=-c_3(\pi_2^*(\cM)^{\vee}-J_1(\P)^{\vee})\cdot \xi_{|X}=-(\pi_2^*(c_3)-6\eta \theta \pi_2^*(c_1)
+(74\eta+2\gamma)\pi_2^*(c_2))\cdot \xi_{|X},$$ for any class
$\xi\in H^2(C\times W^2_{17}(C))$, and
$$c_1^2(U)=c_4(\pi_2^*(\cM)^{\vee}-J_1(\P)^{\vee})=\pi_2^*(c_3)(74\eta+2\gamma)-6\pi_2^*(c_2)\eta \theta.$$
If $\cA_3$ denotes the rank $30$ vector bundle on $X$ having fibres
$$\cA_3(y, M)=H^0(C, L^{\otimes 2})=H^0(C, K_C^{\otimes 2}\otimes
M^{\otimes (-2)}\otimes \OO_C(4y)),$$ then there is an injective
bundle morphism $U^{\otimes 2}\hookrightarrow \cA_3/\cA_2$ and we
consider the quotient sheaf $$\G:=\frac{\cA_3/\cA_2}{U^{\otimes
2}}$$ We note that since the morphism $U^{\otimes 2}\rightarrow
\cA_3/\cA_2$ vanishes along the curve $\Gamma_0$ corresponding to
pairs $(y, M)$ where $M$ has a base point, $\G$ has torsion along
$\Gamma_0$. A straightforward local analysis now shows that
$\F_{|X}$ can be identified as a subsheaf of $\cA_3$ with the kernel
of the map $\cA_3\rightarrow \G$. Therefore, there is an exact
sequence of vector bundles on $X$
$$0\longrightarrow \cA_{2 |X}\longrightarrow \F_{| X}\longrightarrow
U^{\otimes 2}\longrightarrow 0,$$ which over a generic point of $X$
corresponds to the decomposition
$$\F(y, M)=H^0(C, L^{\otimes 2}\otimes \OO_C(-2y))\oplus \mathbb
C\cdot u^2,$$ where $u \in H^0(C, L)$ is such that
$\mbox{ord}_y(u)=1$ (The analysis above, shows that the sequence stays exact over $\Gamma_0$ as well).  Hence
$$c_1(\F_{| X})=c_1(\cA_{2 |X})+2c_1(U)$$
and $c_2(\F_{| X})=c_2(\cA_{2 |X})+2c_1(\cA_{2 |X}) c_1(U)$. Furthermore, we note that the
vector bundle $\pi_2^*\bigl(R^1\pi_{2
*}(\P)\bigr)_{| X}^{\vee}$ is a subbundle of $\E_{|X}$ and we have
an exact sequence
$$0\longrightarrow \pi_2^*\bigl(R^1\pi_{2
*}(\P)\bigr)_{| X}^{\vee}\longrightarrow \E_{|X}\longrightarrow U\longrightarrow
0$$ from which we find that
$c_1(\E_{|X})=-\theta+\pi_2^*(c_1)+c_1(U)$. Similarly, we have that
\begin{equation}\label{c2l}
c_2(\E_{|X})=\frac{\theta^2}{2}+\pi_2^*(c_2)-\theta
\pi_2^*(c_1)-c_1(U)\pi_2^*(c_1)-\theta c_1(U).
\end{equation}
It is elementary to check that $c_1(\mbox{Sym}^2 \E_{|
X})=8 \ c_1(\E_{|X})$ and that $$c_2(\mbox{Sym}^2
\E_{|X})=27\ c_1^2(\E_{|X})+9\ c_2(\E_{|X}),$$ therefore we obtain that
$$
\sigma^*(C^1)\cdot c_2(\F-\mbox{Sym}^2(\E))=c_2(\cA_{2
|X})+c_1(\cA_{2 |X})c_1(U^{\otimes 2})-$$ $$-8c_1(\cA_{2
|X})c_1(\E_{|X})-8c_1(\E_{| X})c_1(U^{\otimes 2})+ 37c_1^2(\E_{|
X})-9c_2(\E_{|X})=
$$
$$=\Bigl(-120\ \eta \theta+\frac{17}{2} \theta^2-16\ \theta
\gamma-9\ \pi_2^*(c_2)+(224\ \eta+32\ \gamma-33\ \theta) \pi_2^*(c_1) +37
\pi_2^*(c_1^2)\Bigr)\cdot [X]+$$ $$+(168\ \eta+24\ \gamma-25
\ \theta+49\ \pi_2^*(c_1))\cdot c_1(U)+21c_1^2(U)=$$
$$=1754\ \eta \theta \pi_2^*(c_2)+1386\ \eta \pi_2^*(c_3)-2498\ \eta
\theta \pi_2^*(c_1^2)+741\ \eta \theta^2 \pi_2^*(c_1)-4068\ \eta
\pi_2^*(c_1)\pi_2^*(c_2)-$$ $$-51\ \eta \theta^3+2738\ \eta
\pi_2^*(c_1^3),$$ where the last expression lives inside
$H^4(C\times W^2_{17}(C))$. Using \cite{F5} Propositions 2.6, each term in this sum
is evaluated and we find that $$\sigma^*(C^1)\cdot
c_2(\F-\mbox{Sym}^2(\E))=691 \ \theta^{21}/1207084032000,$$ which
implies the stated formula for $b_1$.
\end{proof}

\begin{theorem}\label{elltail}
Let $[C, q]\in \cM_{21, 1}$ be a suitably general pointed curve and
$L\in W^6_{25}(C)$ a linear series with a cusp at $q$. Then the
multiplication map $$\mathrm{Sym}^2 H^0(C, L)\rightarrow H^0(C,
L^{\otimes 2})$$ is injective. It follows that we have the relation
$a-12b_0+b_1=0$.
\end{theorem}
\begin{proof} We consider the pencil $R\subset \mm_g$ obtained
by attaching to $C$ at the point $q$ a pencil of plane cubics. It is
well-known that $R\cdot \lambda=1, R\cdot \delta_0=12$ and $R\cdot
\delta_1=-1$, thus the relation $a-12b_0+b_1=0$ would be immediate
once we show that $R\cdot c_2(\F-\mathrm{Sym}^2(\E))=0$.
This follows because of the way the vector bundles $\E$ and $\F$ are defined over the boundary divisor
$\Delta_1^0$ of $\widetilde{\cM}_{22}$, by retaining the aspect of the limit linear series of the component
of genus $21$ and dropping the aspect of the elliptic component.
\end{proof}

\begin{theorem}\label{d0}
Let $[C, q]\in \cM_{21, 1}$ be a Brill-Noether general pointed curve
and denote by $C^0\subset \Delta_0$ the associated test curve. Then
$\sigma^*(C^0)\cdot c_2(\F-\mathrm{Sym}^2(\E))=42b_0-b_1=4847375988$.
It follows that $b_0=132822768$.
\end{theorem}

\begin{proof} This time we look at the virtual degeneracy locus
of the morphism $\mathrm{Sym}^2(\E)\rightarrow \F$ along the surface
$\sigma^*(C^0)$. The first thing to note is that the vector bundles
$\E_{| \sigma^*(C^0)}$ and $\F_{| \sigma^*(C^0)}$ are both
pull-backs of vector bundles on $Y$. For convenience we denote this
vector bundles also by $\E$ and $\F$, hence to use the notation of
Proposition \ref{limitlin0}, $\E_{| \sigma^*(C^0)})=\epsilon^*(\E_{|
Y})$ and $\F_{| \sigma^*(C^0)}=\epsilon^*(\F_{| Y})$. We find that
$$\sigma^*(C^0)\cdot c_2(\F-\mbox{Sym}^2(\E))=c_2(\F_{|Y})-c_1(\F_{| Y})\cdot
c_1(\E_{| Y})+c_1^2(\E_{| Y})-c_2(\E_{|Y})$$ and like in the proof
of Theorem \ref{d1}, we are going to compute each term in this
expression. We denote by $V:=\mbox{Ker}(\chi)$, where $\chi:
\cB^{\vee}\rightarrow \pi_2^*(\cM)^{\vee}$ is the bundle
 morphism on $C\times W^2_{17}(C)$ whose degeneracy locus is $Y$
  and which globalizes all the maps $H^0(C, \OO_{y+q}(M))^{\vee}\rightarrow H^0(C, M)^{\vee}$.
  Thus the kernel bundle $V$ is a line bundle on $Y$ with  fibre
$$V(y, M)=\frac{H^0(C, L)}{H^0(C, L\otimes \OO_C(-y-q))},$$
over each point $(y, M)\in Y$, and where $L:=K_C\otimes
M^{\vee}\otimes \OO_C(y+q)\in W^6_{25}(C)$. By using again the
Harris-Tu Theorem, we find the following formulas for the Chern
numbers of $V$:
$$c_1(V)\cdot \xi_{|Y}=-(c_3(\pi_2^*(\cM)^{\vee}-\cB^{\vee})\cdot \xi_{| Y})=(\pi_2^*(c_3)+\pi_2^*(c_2)(16\eta+\gamma)-2\pi_2^*(c_1)\eta \theta)\cdot \xi_{| Y},$$
for any class $\xi\in H^2(C\times W^2_{17}(C))$, and
$$c_1^2(V)=c_4(\pi_2^*(\cM)^{\vee}-\cB^{\vee})=\pi_2^*(c_3)(16\eta+\gamma)-2\pi_2^*(c_2)\eta \theta.$$
Recall that we have introduced in Proposition \ref{a120} the rank $28$ vector bundle $\cB_2$ over
$C\times W^2_{17}(C)$ with fibre $\cB_2(y, M)=H^0(C, L^{\otimes
2}\otimes \OO_C(-y-q))$. We claim that one has an exact sequence of
bundles over $Y$
\begin{equation} \label{exseq} 0\longrightarrow
\cB_{2 |Y}\longrightarrow \F_{| Y}\longrightarrow V^{\otimes
2}\longrightarrow 0. \end{equation}
 If $\cB_3$ is the rank $30$
vector bundle on $Y$ with fibres
$$\cB_3(y, M)=H^0(C, L^{\otimes 2})=H^0\bigl(C, K_C^{\otimes 2}\otimes M^{\otimes (-2)}\otimes
\OO_C(2y+2q)\bigr),$$  we have an injective morphism of sheaves
$V^{\otimes 2} \hookrightarrow \cB_3/\cB_2$ locally given by
$$v^{\otimes 2}\mapsto v^2 \mbox{ mod } H^0(C, L^{\otimes 2}\otimes
\OO_C(-y-q)),$$ where $v\in H^0(C, L)$ is any section not vanishing
at $q$ and $y$. Then $\F_{|Y}$ is canonically identified with the
kernel of the projection morphism
$$\cB_3\rightarrow \frac{\cB_3/\cB_2}{V^{\otimes 2}}$$
and the exact sequence (\ref{exseq}) now becomes clear. Therefore
$c_1(\F_{| Y})=c_1(\cB_{2 |Y})+2c_1(V)$ and $c_2(\F_{|Y})=c_2(\cB_{2
|Y})+2c_1(\cB_{2|Y}) c_1(V)$. Reasoning along the lines of Theorem \ref{d1}, we also have
an exact sequence
$$0\longrightarrow \pi_2^*\bigl(R^1\pi_{2 *}(\P)\bigr)^{\vee}_{| Y}\longrightarrow \E_{|Y}\longrightarrow V\longrightarrow 0$$
and from this we obtain that
$$c_1(\E_{| Y})=-\theta+\pi_2^*(c_1)+c_1(V)$$
and
$$ c_2(\E_{| Y})=
\frac{\theta^2}{2}+\pi_2^*(c_2)-\theta \pi_2^*(c_1)-\theta c_1(V)+c_1(V) \pi_2^*(c_1).$$ All in all,
we can write the following expression for the total intersection
number:
$$\sigma^*(C^0)\cdot c_2(\F-\mbox{Sym}^2(\E))=c_2(\cB_{2| Y})+c_1(\cB_{2|Y})c_1(V^{\otimes 2})-$$
$$-
8c_1(\cB_{2 |Y})c_1(\E_{|Y})-8c_1(\E_{| Y})c_1(V^{\otimes 2})+37 c_1^2(\E_{|Y})-9c_2(\E_{| Y})=$$
$$=\Bigl(\frac{17}{2}\theta^2+28\eta \theta-8\theta \gamma-9\pi_2^*(c_2)+(16\gamma-33\theta-56\eta)\pi_2^*(c_1)+37\pi_2^*(c_1^2)\Bigr)\cdot [Y]+$$
$$+(49\ \pi_2^*(c_1)-25\ \theta-42\ \eta+12\ \gamma)c_1(V)+21 c_1^2(V)=$$
$$=428 \ \eta \theta \pi_2^*(c_2)-536\  \eta \theta \pi_2^*(c_1^2)+168
\  \eta \theta^2 \pi_2^*(c_1)-984 \ \eta \pi_2^*(c_1)
\pi_2^*(c_2)+$$ $$+378 \eta \pi_2^*(c_3)-17 \ \eta \theta^3 + 592
\eta \pi_2^*(c_1^3),$$ and using once more \cite{F5} Proposition 2.6,
we get that $$42 b_0-b_1=509 \theta^{21}/5364817920000.$$ Since we
already know the value of $b_1$ and $a-12b_0+b_1=0$, this allows us
to  calculate $a$ and $b_0$.
\end{proof}
\noindent
\emph{End of the proof of Theorem \ref{m22}}. We write
$\overline{\mathfrak{D}}_{22}\equiv a \lambda-\sum_{j=0}^{11} b_j
\delta_j$. Since $$\frac{a}{b_0}=\frac{17121}{2636}\leq \frac{71}{10},$$ we are in a position
to apply Corollary 1.2 from \cite{FP} which gives the inequalities
$b_j\geq b_0$ for $1\leq j\leq 11$, hence
$s(\overline{\mathfrak{D}}_{22})=a/b_0<13/2$.\hfill $\Box$

\end{document}